\documentclass{amsart}

\usepackage[T1]{fontenc}
\usepackage{enumerate, amsmath, amsfonts, amssymb, amsthm, thmtools, mathrsfs, wasysym, graphics, graphicx, xcolor, url, hyperref, hypcap, xargs, multicol, pdflscape, multirow, hvfloat, array, ae, aecompl, pifont, mathtools, a4wide, float, blkarray, overpic, nicefrac, etoolbox, stmaryrd, shuffle}
\usepackage[bb=boondox]{mathalfa}
\usepackage[shortlabels, inline]{enumitem}
\usepackage[noabbrev,capitalise]{cleveref}
\usepackage[normalem]{ulem}
\usepackage{marginnote}
\hypersetup{colorlinks=true, citecolor=darkblue, linkcolor=darkblue}
\usepackage[all]{xy}
\usepackage{tikz}
\usepackage{tikz-cd}
\usetikzlibrary{trees, decorations, decorations.pathmorphing, decorations.markings, decorations.shapes, shapes, arrows, matrix, calc, fit, intersections, patterns, angles}
\graphicspath{{figures/}{figures/nodes/}}
\makeatletter\def\input@path{{figures/}}\makeatother
\usepackage{caption}
\usepackage[export]{adjustbox}

\newtheorem{theorem}{Theorem}[section]
\newtheorem{corollary}[theorem]{Corollary}
\newtheorem{proposition}[theorem]{Proposition}
\newtheorem{lemma}[theorem]{Lemma}

\newtheorem*{theorem*}{Theorem}

\theoremstyle{definition}
\newtheorem{definition}[theorem]{Definition}
\newtheorem{example}[theorem]{Example}
\newtheorem{remark}[theorem]{Remark}

\crefname{equation}{Equation}{Equations}

\newcommand{\R}{\mathbb{R}} 
\newcommand{\f}[1]{{\mathfrak{#1}}} 
\renewcommand{\c}[1]{{\mathcal{#1}}} 
\renewcommand{\b}[1]{{\boldsymbol{#1}}} 

\newcommand{\set}[2]{\left\{ #1 \;\middle|\; #2 \right\}} 
\newcommand{\bigset}[2]{\big\{ #1 \;\big|\; #2 \big\}} 
\newcommand{\ssm}{\smallsetminus} 
\newcommand{\one}{{\mathbb{1}}} 
\newcommand{\shiftedShuffle}{\,\bar\shuffle\,} 

\newcommand{\eqdef}{\mbox{\,\raisebox{0.2ex}{\scriptsize\ensuremath{\mathrm:}}\ensuremath{=}\,}} 

\DeclareMathOperator{\inv}{inv} 
\DeclareMathOperator{\des}{des} 
\DeclareMathOperator{\tors}{\mathsf{tors}}

\newcommand{\ie}{\textit{i.e.}~} 
\newcommand{\apriori}{\textit{a~priori}~} 
\definecolor{darkblue}{rgb}{0,0,0.7} 
\definecolor{green}{RGB}{57,181,74} 
\definecolor{violet}{RGB}{147,39,143} 
\newcommand{\red}{\color{red}} 
\newcommand{\darkblue}{\color{darkblue}} 
\newcommand{\defn}[1]{\textsl{\darkblue #1}} 

\usepackage{todonotes}

\newcommandx{\polytope}[1][1=P]{\mathsf{#1}} 
\newcommandx{\Perm}[1][1=n]{\polytope[Perm]_{#1}} 
\newcommandx{\Asso}[1][1=n]{\polytope[Asso]_{#1}} 
\newcommandx{\Para}[1][1=n]{\polytope[Para]_{#1}} 
\newcommandx{\Zono}[1][1=\digraph]{\polytope[Zono]_{#1}} 
\newcommandx{\Quotientope}[1][1=\equiv]{\polytope[Quot]_{#1}} 
\newcommandx{\Fan}[1][1=F]{\mathcal{#1}} 
\newcommandx{\braidFan}[1][1=n]{\Fan_{#1}} 
\newcommandx{\quotientFan}[1][1=\equiv]{\c{F}_{#1}} 
\newcommand{\HA}{\mathcal{H}} 
\newcommand{\hyp}{\mathbb{H}} 
\newcommandx{\ray}[1][1=r]{\boldsymbol{#1}} 
\newcommandx{\Cone}[1][1=C]{\polytope[#1]} 

\newcommand{\arc}{\alpha} 
\newcommand{\arcs}{{\mathcal{A}}} 
\newcommand{\meet}{\wedge} 
\newcommand{\join}{\vee} 
\newcommand{\decoration}{{\b{\delta}}} 
\newcommand{\includeSymbol}[1]{\ensuremath{%
	\mathchoice
		{\raisebox{-.7mm}{\includegraphics[height=2.2ex]{#1}}}	
		{\raisebox{-.7mm}{\includegraphics[height=2.2ex]{#1}}}
		{\raisebox{-.6mm}{\includegraphics[height=1.6ex]{#1}}}
		{\raisebox{-.5mm}{\includegraphics[height=1ex]{#1}}}
}}
\robustify{\includeSymbol}
\newcommand{\noneCirc}{\includeSymbol{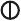}}
\newcommand{\upCirc}{\includeSymbol{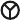}}
\newcommand{\downCirc}{\includeSymbol{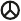}}
\newcommand{\upDownCirc}{\includeSymbol{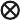}}
\newcommand{\Decorations}{\{\noneCirc{}, \downCirc{}, \upCirc{}, \upDownCirc{}\}} 
\newcommand{\projDown}{\smash{\pi^{\equiv}_{\!\downarrow\!}}} 
\newcommand{\projUp}{\smash{\pi_{\equiv}^{\!\uparrow\!}}} 
\newcommand{\arcDown}{\smash{\alpha_{\downarrow\!}}} 
\newcommand{\arcUp}{\smash{\alpha^{\uparrow\!}}} 
\newcommandx{\arcDiagramDown}[1][1=\equiv]{\smash{\delta^{#1}_{\downarrow\!}}} 
\newcommandx{\arcDiagramUp}[1][1=\equiv]{\smash{\delta_{#1}^{\uparrow\!}}} 

\newcommand{\RA}{RA} 

\makeatletter
\def\part{\@startsection{part}{1}%
\z@{.7\linespacing\@plus\linespacing}{.8\linespacing}%
{\LARGE\sffamily\centering}}
\makeatother

\makeatletter
\def\l@section{\@tocline{1}{5pt}{0pc}{}{}}
\makeatother
\let\oldtocpart=\tocpart
\renewcommand{\tocpart}[2]{\sc\large\oldtocpart{#1}{#2}}
\let\oldtocsection=\tocsection
\renewcommand{\tocsection}[2]{\bf\oldtocsection{#1}{#2}}
\let\oldtocsubsubsection=\tocsubsubsection
\renewcommand{\tocsubsubsection}[2]{\quad\oldtocsubsubsection{#1}{#2}}

\title{Separating trees and simple congruences of the weak order}

\thanks{
JCN was partially supported by the project CARPLO of the Agence Nationale de la recherche (ANR-20-CE40-0007).
VP was partially supported by the Spanish project PID2022-137283NB-C21 of MCIN/AEI/10.13039/501100011033 / FEDER, UE, by the Spanish--German project COMPOTE (AEI PCI2024-155081-2 \& DFG 541393733), by the Severo Ochoa and María de Maeztu Program for Centers and Units of Excellence in R\&D (CEX2020-001084-M), by the Departament de Recerca i Universitats de la Generalitat de Catalunya (2021 SGR 00697), and by the French--Austrian project PAGCAP (ANR-21-CE48-0020 \& FWF I 5788).
}

\author{Emily Barnard}
\address[Emily Barnard]{DePaul University, Chicago, U.S.A.}
\email{e.barnard@depaul.edu}
\urladdr{\url{https://emilybarnard.github.io}}

\author{Jean-Christophe Novelli}
\address[Jean-Christophe Novelli]{LIGM, Universit\'e Gustave-Eiffel, CNRS,
ENPC, ESIEE-Paris \\
5 Boulevard Descartes \\Champs-sur-Marne \\77454 Marne-la-Vall\'ee cedex 2 \\
FRANCE}
\email{jean-christophe.novelli@univ-eiffel.fr}
\urladdr{\url{https://igm.univ-mlv.fr/~novelli/}}

\author{Vincent Pilaud}
\address[Vincent Pilaud]{Universitat de Barcelona \& Centre de Recerca Matemàtica, Barcelona, Spain}
\email{vincent.pilaud@ub.edu}
\urladdr{\url{https://www.ub.edu/comb/vincentpilaud/}}

\begin{document}

\begin{abstract}
A congruence of the weak order is simple if its quotientope is a simple polytope.
We provide an alternative elementary proof of the characterization of the simple congruences in terms of forbidden up and down arcs.
For this, we provide a combinatorial description of the vertices of the corresponding quotientopes in terms of separating trees.
This also yields a combinatorial description of all faces of the corresponding quotientopes.
We finally explore algebraic aspects of separating trees, in particular their connections with quiver representation theory.
\end{abstract}

\maketitle

\tableofcontents

\pagebreak


\section{Introduction}
\label{sec:introduction}

Lattice congruences of the weak order of a Coxeter group were pioneered by N.~Reading~\cite{Reading-CambrianLattices} in connection to the combinatorics of finite type cluster algebras~\cite{FominZelevinsky-ClusterAlgebrasI,FominZelevinsky-ClusterAlgebrasII}.
For the symmetric group, N.~Reading described the combinatorics of a lattice congruence~$\equiv$ of the weak order in terms of noncrossing arc diagrams~\cite{Reading-arcDiagrams}.
He also studied in~\cite{Reading-HopfAlgebras} the quotient fan~$\quotientFan$ obtained by coarsening the braid arrangement according to the congruence classes of~$\equiv$.
Moreover, V.~Pilaud and F.~Santos~\cite{PilaudSantos-quotientopes, PadrolPilaudRitter} constructed the quotientope~$\Quotientope$, whose normal fan is the quotient fan~$\quotientFan$, and whose oriented skeleton is the quotient~$\f{S}_n/{\equiv}$.
The most classical congruence of the weak order is certainly the sylvester congruence~\cite{Tonks, HivertNovelliThibon-algebraBinarySearchTrees}, whose quotient is the Tamari lattice~\cite{Tamari}, and whose quotientope is the associahedron of~\cite{ShniderSternberg, Loday, PilaudSantosZiegler}.

A congruence~$\equiv$ is called simple when the Hasse diagram of the quotient~$\f{S}_n/{\equiv}$ is regular, or equivalently when the quotient fan~$\quotientFan$ is simplicial, or equivalently when the quotientope~$\Quotientope$ is a simple polytope.
Important examples include the sylvester congruence~\cite{Tonks, HivertNovelliThibon-algebraBinarySearchTrees}, the recoil congruence, the Cambrian and permutree congruences~\cite{Reading-CambrianLattices, PilaudPons-permutrees}, and the sashes congruences~\cite{Law,LaniniNovelli}.
These congruences naturally appear in algebraic contexts, in particular in the construction of certain combinatorial Hopf algebras~\cite{LodayRonco, HivertNovelliThibon-algebraBinarySearchTrees, ChatelPilaud, PilaudPons-permutrees}, and in quiver representation theory~\cite{DemonetIyamaReadingReitenThomas}.

The simple congruences of the weak order are precisely those whose minimal contracted arcs do not cross the horizontal axis.
The forward direction of this statement is easy to prove.
The backward direction was originally proved by H.~Hoang and T.~Mütze in~\cite[Sect.~4.4]{HoangMutze} with a purely combinatorial (but slightly involved) method.
In their paper on the lattice theory of torsion classes, L.~Demonet, O.~Iyama, N.~Reading, I.~Reiten, and H.~Thomas provided an alternative proof~\cite[Sect.~6.3]{DemonetIyamaReadingReitenThomas} in terms of quiver representation theory.

Our first contribution is an alternative elementary combinatorial proof of this characterization based on the description of the Hasse diagrams of the posets of a simple congruence~$\equiv$.
Namely, we show that they are \defn{separating trees}, that is, directed trees on~$[n]$ where each node has at most two parents (\ie outgoing neighbors) and at most two children (\ie incoming neighbors), and where any node with two parents (resp.~children) separates its ancestor (resp.~descendant) subtrees, \ie the node is larger than all nodes of the left ancestor (resp.~descendant) subtree and smaller than all nodes of the right ancestor (resp.~descendant) subtree.
Our proof is based on the insertion map from permutations to posets of~$\equiv$ already described in~\cite{Pilaud-arcDiagramAlgebra}.

In fact, we actually completely characterize which separating trees appear for a given simple congruence~$\equiv$.
Namely, given a separating tree~$T$, each edge of~$T$ defines a mandatory arc while each pair of parents or children of a node of~$T$ defines a forbidden arc.
We prove that the separating tree~$T$ appears for the congruence~$\equiv$ if and only if~$\equiv$ contracts the forbidden arcs of~$T$ but not the mandatory arcs of~$T$.
In particular, the congruences in which a given separating tree appears form an interval in the lattice of simple congruences.

We then consider Schröder separating trees, that is, directed trees on a partition of~$[n]$ such that any two ancestor (resp.~descendant) subtrees of a node~$J$ are separated by at least one element of~$J$. 
They can also be obtained by edge contractions in separating trees, so that the Hasse diagrams of the preposets of a simple congruence all are Schröder separating trees.
Again, we characterize which Schröder separating trees appear for a given simple congruence~$\equiv$ in terms of mandatory and forbidden arcs contracted by~$\equiv$.
Let us underline that this provides the first combinatorial model for all faces of the quotientope~$\Quotientope$ when~$\equiv$ is a simple congruence.
Note that for an arbitrary congruence~$\equiv$, there are combinatorial models for the vertices of the quotientope~$\Quotientope$ by~\cite{Reading-arcDiagrams}, for the facets of~$\Quotientope$ by~\cite{AlbertinPilaudRitter}, and for the intervals in the quotient~$\f{S}_n/{\equiv}$ by~\cite{AlbertinPilaud}, but there is no combinatorial model for all faces of the quotientope~$\Quotientope$.

Finally, we discuss some algebraic aspects of separating trees.
In \cref{sec:algebraicStructure}, we define a natural algebra structure on decorated separating trees similar to that of~\cite{LodayRonco, ChatelPilaud, PilaudPons-permutrees}, but which unfortunately fails to define a Hopf algebra structure on decorated separating trees.
In \cref{sec:representationTheory}, we rephrase our result in terms of algebraic quotients of the lattice of torsion classes of the preprojective algebra of type~$A_n$, already largely explored in~\cite{DemonetIyamaReadingReitenThomas}.


\section{Lattice congruences of the weak order}
\label{sec:latticeCongruences}

We first recall the combinatorial and geometric toolbox to deal with lattice congruences of the weak order on~$\f{S}_n$.
The presentation and pictures are borrowed from~\cite{PilaudSantos-quotientopes, PadrolPilaudRitter}.
Throughout the paper, we let~$[n] \eqdef \{1, \dots, n\}$, $[i,j] \eqdef \{i, \dots, j\}$ and~${{]i,j[} \eqdef \{i+1, \dots, j-1\}}$ for~$i < j$.


\subsection{Weak order, braid fan, permutahedron}
\label{subsec:weakOrder}

The \defn{weak order} is the lattice on permutations of~$[n]$ defined by ${\sigma \le \tau \iff \inv(\sigma) \subseteq \inv(\tau)}$ where~$\inv(\sigma) \eqdef \set{(\sigma_p, \sigma_q)}{1 \le p < q \le n \text{ and } \sigma_p > \sigma_q}$ is the \defn{inversion set} of~$\sigma$.
The cover relations correspond to exchanges of two entries in adjacent positions in a permutation: a permutation~$\sigma$ covers (resp.~is covered by) the permutations obtained from~$\sigma$ by reversing a \defn{descent}~$\sigma_p > \sigma_{p+1}$ (resp.~an \defn{ascent}~$\sigma_p < \sigma_{p+1}$) of~$\sigma$.
See \cref{fig:weakOrder4}\,(left).

The \defn{braid arrangement} is the set~$\HA_n$ of hyperplanes~$\set{\b{x} \in \R^n}{\b{x}_i = \b{x}_j}$ for ${1 \le i < j \le n}$.
They divide~$\R^n$ into chambers, which are the maximal cones of a complete simplicial fan~$\braidFan$, called \defn{braid fan}.
This fan has a cone~$\Cone(\mu) \eqdef \set{\b{x} \in \R^n}{x_i \le x_j \text{ for all } i \preccurlyeq_\mu j}$ for each ordered set partition~$\mu$ of~$[n]$.
Here~$\preccurlyeq_\mu$ denotes the \defn{preposet} (\ie reflexive and transitive binary relation) on~$[n]$ defined by~$i \preccurlyeq_\mu j$ if the part of~$\mu$ containing~$i$ is before or equal to the part of~$\mu$ containing~$j$.
In particular, $\braidFan$ has a maximal cone for each permutation of~$[n]$.
See \cref{fig:weakOrder4}\,(middle).

The \defn{permutahedron} is the polytope~$\Perm$ defined as the convex hull of the points~$\sum_{i \in [n]} i \, \b{e}_{\sigma_i}$ for all permutations~$\sigma \in \f{S}_n$.
See \cref{fig:weakOrder4}\,(right).
The normal fan of the permutahedron~$\Perm$ is the braid fan~$\braidFan$.
The Hasse diagram of the weak order on~$\f{S}_n$ can be seen geometrically as the dual graph of the braid fan~$\braidFan$, or the graph of the permutahedron~$\Perm$, oriented in the linear direction~$\b{\alpha} \eqdef \sum_{i \in [n]} (2i-n-1) \, \b{e}_i = (-n+1, -n+3, \dots, n-3, n-1)$.
See \cref{fig:weakOrder4}.

\begin{figure}
	\capstart
	\centerline{\includegraphics[scale=.6]{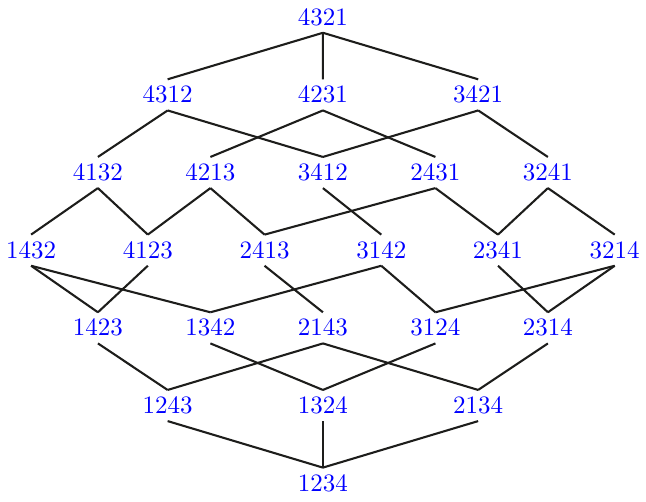} \; \includegraphics[scale=.6]{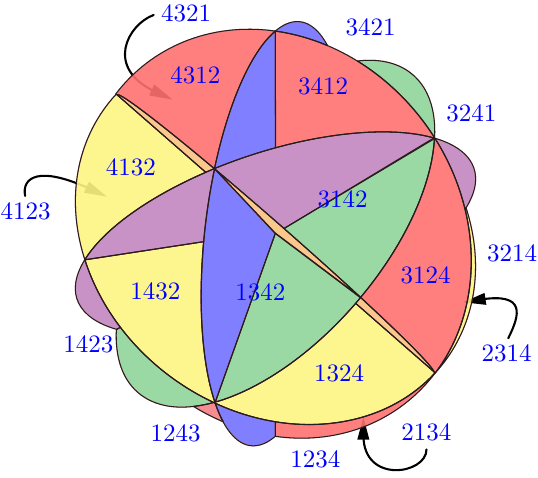} \; \includegraphics[scale=.6]{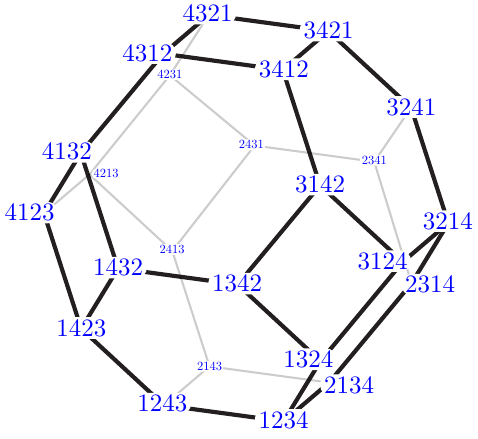}}
	\caption{The Hasse diagram of the weak order on~$\f{S}_4$ (left) can be seen as the dual graph of the braid fan~$\braidFan[4]$ (middle) or as the graph of the permutahedron~$\Perm[4]$ (right). \cite[Fig.~1]{PilaudSantos-quotientopes}}
	\label{fig:weakOrder4}
\end{figure}


\subsection{Noncrossing arc diagrams}
\label{subsec:noncrossingArcDiagrams}

We recall from~\cite{Reading-arcDiagrams} the bijection between permutations and noncrossing arc diagrams.
An \defn{arc} is a quadruple~$(i, j, A, B)$ formed by two integers $1 \le i < j \le n$ and a partition~${A \sqcup B = {]i,j[}}$.
We represent it by a curve joining~$i$ to~$j$, wiggling around the horizontal axis, passing above the points of~$A$ and below the points of~$B$.
For instance, \smash{\raisebox{-.13cm}{\includegraphics[scale=.8]{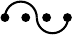}}} represents the arc~$(1,4,\{2\},\{3\})$.
We let~$\arcs_n \eqdef \set{(i, j, A, B)}{1 \le i < j \le n \text{ and } A \sqcup B = {]i,j[}}$ denote the set of all arcs.

Two arcs~$(i, j, A, B)$ and~$(i', j', A', B')$ \defn{cross} if the two corresponding curves intersect in their interior, that is, if both~$(A \cap B') \cup (\{i,j\} \cap B') \cup (A \cap \{i',j'\})$ and~$(B \cap A') \cup (\{i,j\} \cap A') \cup (B \cap \{i',j'\})$ are non-empty.
A \defn{noncrossing arc diagram} is a collection of arcs of~$\arcs_n$ where any two arcs do not cross and have distinct left endpoints and distinct right endpoints (but the right endpoint of an arc can be the left endpoint of another arc).
See \cref{fig:noncrossingArcDiagrams} for illustrations.

Consider now a permutation~$\sigma$ of~$[n]$, and its \defn{permutation table} containing the points~$(\sigma_p, p)$ for~$p \in [n]$.
(This unusual choice of orientation is necessary to fit later with the existing constructions ~\cite{LodayRonco, HivertNovelliThibon-algebraBinarySearchTrees, ChatelPilaud, PilaudPons-permutrees}.)
We draw segments between any two consecutive dots~$(\sigma_p, p)$ and~$(\sigma_{p+1}, p+1)$, colored blue if~$\sigma_p < \sigma_{p+1}$ is an \defn{ascent} and red if~$\sigma_p > \sigma_{p+1}$ is a \defn{descent}.
We then let all dots fall down to the horizontal axis, allowing the segments to curve, but not to cross each other nor to pass through any dot, as illustrated in \cref{fig:noncrossingArcDiagrams}.
The resulting set of blue (resp.~red) arcs is a noncrossing arc diagram~$\arcDiagramUp[](\sigma)$ (resp.~$\arcDiagramDown[](\sigma)$).

\begin{theorem}[{\cite[Thm.~3.1]{Reading-arcDiagrams}}]
The maps~$\arcDiagramUp[](\sigma)$ and~$\arcDiagramDown[](\sigma)$ are bijections from the permutations of~$\f{S}_n$ to the noncrossing arc diagrams on~$\arcs_n$.
\end{theorem}

The map~$\arcDiagramUp[]$ (resp.~$\arcDiagramDown[]$) is particularly adapted to understand canonical meet (resp.~join) representations in the weak order.
Namely, each single arc corresponds to a meet (resp.~join) irreducible permutation, \ie with a single ascent (resp.~descent), and the arc diagram~$\arcDiagramUp[](\sigma)$ (resp.~$\arcDiagramDown[](\sigma)$) of a permutation corresponds to the canonical meet (resp.~join) representation of~$\sigma$.
See \cite{Reading-arcDiagrams} for more details.
While the description of lattice quotients of the weak order below is based on this fact, we do not strictly need the precise details in this paper.

\begin{figure}
	\capstart
	\centerline{\includegraphics[scale=.85]{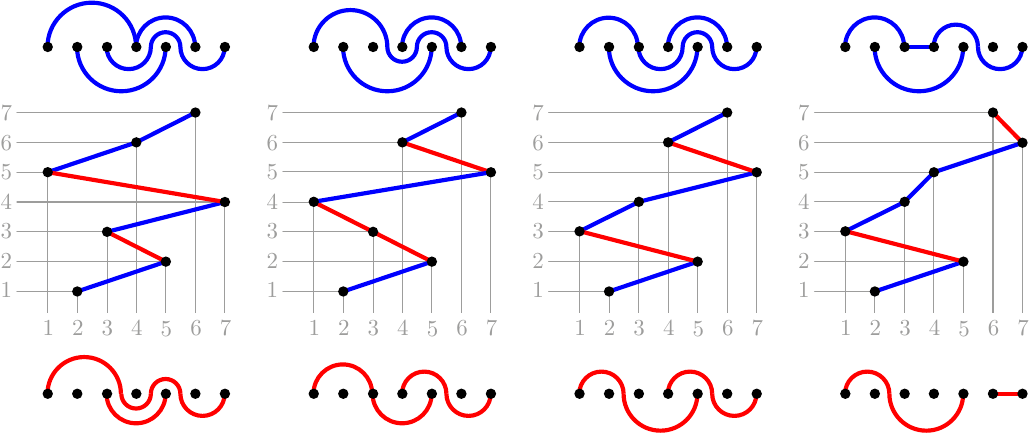}}
	\caption{The noncrossing arc diagrams~$\arcDiagramDown[](\sigma)$ (bottom) and~$\arcDiagramUp[](\sigma)$ (top) for the permutations~$\sigma = 2537146$, $2531746$, $2513746$, and $2513476$. \cite[Fig.~2]{Pilaud-arcDiagramAlgebra}}
	\label{fig:noncrossingArcDiagrams}
\end{figure}


\subsection{Lattice congruences and arc ideals}
\label{subsec:latticeCongruencesArcIdeals}

A \defn{lattice congruence} of the weak order is an equivalence relation~$\equiv$ on~$\f{S}_n$ that respects the meet and join operations, \ie such that $x \equiv x'$ and~$y \equiv y'$ implies $x \meet y \, \equiv \, x' \meet y'$ and~$x \join y \, \equiv \, x' \join y'$.
The \defn{lattice quotient}~$\f{S}_n/{\equiv}$ is the lattice on the congruence classes of~$\equiv$ where~$X \le Y$ if and only if there exist~$x \in X$ and~$y \in Y$ such that~$x \le y$, and~$X \meet Y$ (resp.~$X \join Y$) is the congruence class of~$x \meet y$ (resp.~$x \join y$) for any~$x \in X$~and~$y \in Y$.

The compatibility with the meet and join operations in fact forces all congruence classes of~$\equiv$ to be intervals (indeed, $x \equiv y$ implies that~$x \meet y \equiv x \meet x = x = x \join x \equiv x \join y$ so that congruence classes are stable by meet and join).
For a permutation~$\sigma$, we denote by~$\projDown(\sigma)$ and~$\projUp(\sigma)$ the top and bottom permutations of the $\equiv$-congruence class of~$\sigma$, and we define the noncrossing arc diagrams~$\arcDiagramDown(\sigma) \eqdef \arcDiagramDown[](\projDown(\sigma))$ and~$\arcDiagramUp(\sigma) \eqdef \arcDiagramUp[](\projUp(\sigma))$.
See \cref{fig:arcDiagramsQuotients}.
Note that we will describe later in \cref{subsec:insertionMap} the noncrossing arc diagrams ~$\arcDiagramDown(\sigma)$ and~$\arcDiagramUp(\sigma)$ directly from the permutation~$\sigma$, without passing through the minimal and maximal permutations~$\projDown(\sigma)$ and~$\projUp(\sigma)$ of the congruence class of~$\sigma$.
The interval~$[\projDown(\sigma), \projUp(\sigma)]$ is the set of all linear extensions of a poset~$\preccurlyeq^\equiv_\sigma$ on~$[n]$ whose cover relations are~$j \preccurlyeq^\equiv_\sigma i$ for each descent~$ji$ of~$\projDown(\sigma)$ and~$i \preccurlyeq^\equiv_\sigma j$ for each ascent~$ij$ of~$\projUp(\sigma)$.
Hence, the Hasse diagram of the poset~$\preccurlyeq^\equiv_\sigma$ is obtained by superimposing the arcs of~$\arcDiagramDown(\sigma)$ oriented from right to left with the arcs of~$\arcDiagramUp(\sigma)$ oriented from left to right.

We now denote by~$\arcs_\equiv$ the set of arcs corresponding via~$\arcDiagramDown[]$ to permutations with a single descent and minimal in their class, or equivalently, via~$\arcDiagramUp[]$ to permutations with a single ascent and maximal in their class.
This set~$\arcs_\equiv$ enables us to completely describe the lattice quotient~$\f{S}_n / {\equiv}$.

\pagebreak
\begin{theorem}[{\cite[Thm.~4.1]{Reading-arcDiagrams}}]
The quotient~$\f{S}_n/{\equiv}$ is isomorphic (as a poset) to the subposet of the weak order induced by the permutations which are minimal (resp.~maximal) in their class, which are precisely the permutations~$\sigma$ for which all arcs in the noncrossing arc diagram~$\arcDiagramDown[](\sigma)$ (resp.~$\arcDiagramUp[](\sigma)$) belong~to~$\arcs_\equiv$.
\end{theorem}

This proposition actually translates the fact that canonical representations behave properly under lattice quotients.
Namely, the noncrossing arc diagrams of~$\arcs_\equiv$ precisely encode the canonical join (resp.~meet) representations in the quotient~$\f{S}_n/{\equiv}$.
See~\cite{Reading-arcDiagrams} for more details.

Understanding the lattice congruences of the weak order thus boils down to understand the sets of arcs~$\arcs_\equiv$.
An arc~$\arc \eqdef (i, j, A, B) \in \arcs_n$ is a \defn{subarc} of an arc~$\arc' \eqdef (i', j', A', B') \in \arcs_n$ if~$i' \le i < j \le j'$ and~${A \subseteq A'}$ and~${B \subseteq B'}$.
Visually, $\arc$ is a subarc of~$\arc'$ if the endpoints of~$\arc$ are located in between those of~$\arc'$ and~$\arc$ agrees with~$\arc'$ in between its endpoints.
The subarc poset on~$\arcs_4$ is illustrated in \cref{fig:subarcOrder}.
Note that contrarily to~\cite{Reading-arcDiagrams, PilaudSantos-quotientopes, PadrolPilaudRitter}, we do not reverse the subarc poset to avoid confusion when speaking about subarc minimal arcs (hence, our subarc minimal arcs have minimal length).
An \defn{arc ideal} is a lower ideal of the subarc poset (meaning that if~$\alpha$ is a subarc of~$\beta$, then~$\beta \in \c{I}$ implies~$\alpha \in \c{I}$).

\begin{figure}
	\capstart
	\centerline{\includegraphics[scale=.7,valign=c]{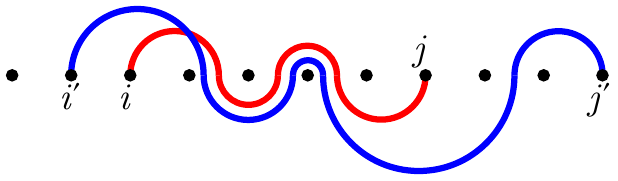} \hspace{1cm} \includegraphics[scale=.6,valign=c]{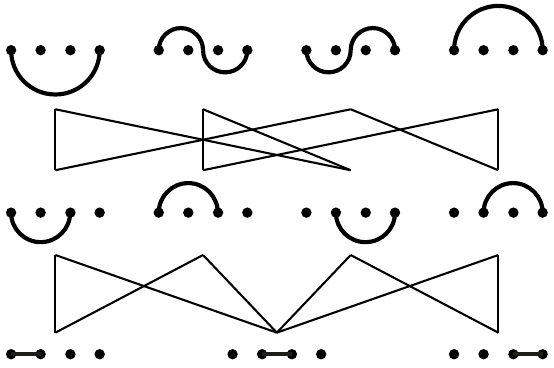}}
	\caption{The subarc relation among arcs (left) and the subarc poset for~$n = 4$ (right). The red arc~$(i,j,A,B)$ is a subarc of the blue arc~$(i',j',A',B')$. \mbox{\cite[Fig.~5]{PilaudSantos-quotientopes}}}
	\label{fig:subarcOrder}
\end{figure}

\begin{theorem}[{\cite[Thm.~4.4 \& Coro.~4.5]{Reading-arcDiagrams}}]
The map~${\equiv} \mapsto \arcs_\equiv$ is a bijection between the lattice congruences of the weak order on~$\f{S}_n$ and the arc ideals of~$\arcs_n$.
\end{theorem}

We denote by~$\equiv_\arcs$ the lattice congruence of the weak order on~$\f{S}_n$ corresponding to an arc ideal~$\arcs \subseteq \arcs_n$.
We make no distinction between arc ideals and lattice congruences.

Finally, note that the set of all congruences of the weak order is itself a lattice, whose meet is the intersection of congruences and join is the transitive closure of union of congruences.
This lattice is distributive, and we have~$\arcs_{{\equiv} \join {\equiv'}} = \arcs_\equiv \cup \arcs_{\equiv'}$ and~$\arcs_{{\equiv} \meet {\equiv'}} = \arcs_\equiv \cap \arcs_{\equiv'}$.

Throughout the paper, we focus on \defn{essential} lattice congruences~$\equiv$, whose arc ideal~$\arcs_\equiv$ contains all basic arcs~$(i, i+1, \varnothing, \varnothing)$ for~$i \in [n-1]$.
The non-essential lattice congruences can be understood as Cartesian products of essential congruences of weak orders of smaller rank.

\begin{example}
\label{exm:sylvesterCongruence}
\begin{figure}
	\capstart
	\centerline{\includegraphics[scale=.6]{weakOrderLeft4} \; \includegraphics[scale=.6]{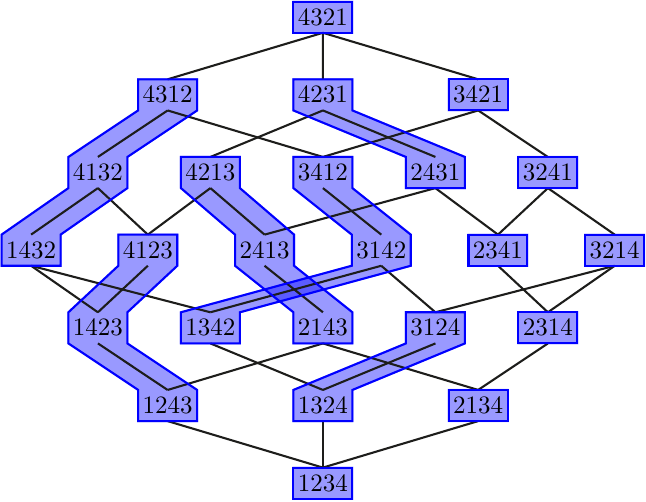} \; \includegraphics[scale=.48]{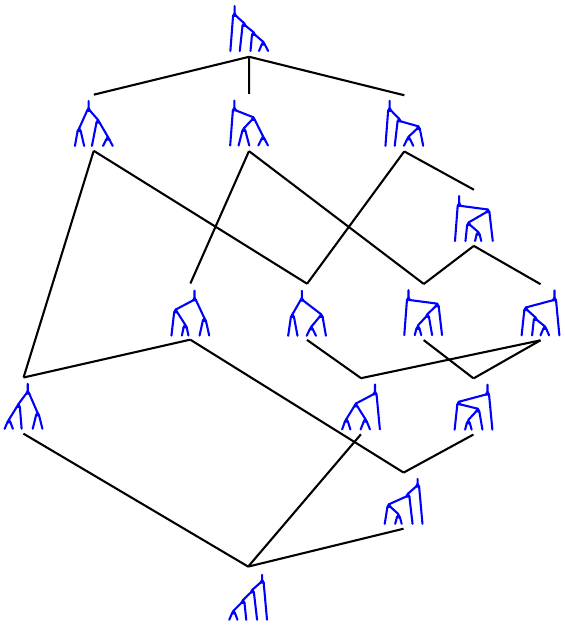}}
	\caption{The weak order on~$\f{S}_4$ (left), the sylvester congruence~$\equiv_\textrm{sylv}$~(middle), and the Tamari lattice (right). \cite[Fig.~1 \& 2]{PilaudSantos-quotientopes}}
	\label{fig:sylvesterCongruence}
\end{figure}
The most classical lattice congruence of the weak order on~$\f{S}_n$ is the \defn{sylvester congruence}~$\equiv_\textrm{sylv}$ \cite{LodayRonco, HivertNovelliThibon-algebraBinarySearchTrees}.
See \cref{fig:sylvesterCongruence}\,(middle).
Its congruence classes are the fibers of the binary search tree insertion algorithm, or equivalently the sets of linear extensions of binary trees, labeled in inorder (\ie a node is labeled after its left subtree and before its right subtree) and considered as posets oriented toward their roots.
It can also be seen as the transitive closure of the rewriting rule~$U i k V j W \equiv_\textrm{sylv} U k i V j W$ where~$i < j < k$ are letters and~$U,V,W$ are words on~$[n]$.
In other words, the permutations minimal in their sylvester congruence class are those avoiding the pattern~$312$.
The quotient of the weak order by the sylvester congruence is (isomorphic to) the classical \defn{Tamari lattice}~\cite{Tamari}, whose elements are the binary trees on~$n$ nodes and whose cover relations are right rotations in binary trees.
See \cref{fig:sylvesterCongruence}\,(right).
Its arc ideal is given by the set~$\arcs_\textrm{sylv} \eqdef \set{(i, j, {]i,j[}, \varnothing)}{1 \le i < j \le n}$ of up arcs, and the corresponding noncrossing arc diagrams are known as \defn{noncrossing partitions} of~$[n]$.
\end{example}


\pagebreak
\subsection{Poset insertion map}
\label{subsec:insertionMap}

\begin{figure}
	\capstart
	\centerline{\includegraphics[scale=.85]{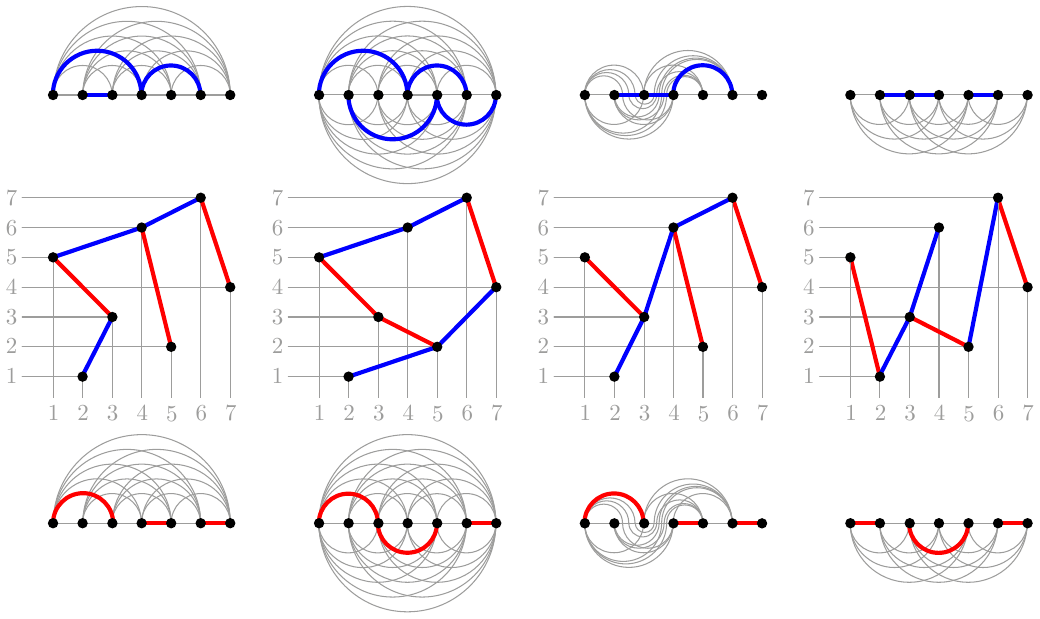}}
	\caption{The poset~$\preccurlyeq_\sigma^\equiv$ (middle row) and the noncrossing arc diagrams~$\arcDiagramUp(\sigma)$ (top row) and~$\arcDiagramDown(\sigma)$ (bottom row) corresponding to the $\equiv_\arcs$-congruence class of the permutation~$\sigma = 2537146$ for different arc ideals~$\arcs$ (represented in grey on the top and bottom rows).}
	\label{fig:arcDiagramsQuotients}
\end{figure}

For a fixed lattice congruence~$\equiv$ of the weak order, we describe in this section the lattice morphism from the weak order to its quotient~$\f{S}/{\equiv}$, sending a permutation~$\sigma$ of~$\f{S}_n$ to its $\equiv$-poset~$\preccurlyeq_\sigma^\equiv$.
We still represent a permutation~$\sigma$ by its table~$\set{(\sigma_p,p)}{p \in [n]} = \set{(i, \sigma^{-1}_i)}{i \in [n]}$.
We consider the set~$\boxslash(\sigma)$ of pairs of values~$(i,j)$ such that the rectangle with bottom left corner~$(i, \sigma^{-1}_i)$ and top right corner~$(j, \sigma^{-1}_j)$ contains no other point~$(k, \sigma^{-1}_k)$ of the permutation table of~$\sigma$, that is,
\[
\boxslash(\sigma) \eqdef \set{(i, j)}{1 \le i < j \le n, \, \sigma^{-1}_i < \sigma^{-1}_j \text{ and } \sigma^{-1}({]i,j[}) \cap {]\sigma^{-1}_i, \sigma^{-1}_j[} = \varnothing}.
\]
We order~$\boxslash(\sigma)$ by the slope of the diagonal joining~$(i, \sigma^{-1}_i)$ to~$(j, \sigma^{-1}_j)$, that is, $(i, j) \trianglelefteq (k,\ell)$ if~$i \le k < \ell \le j$ and~$\sigma^{-1}_k \le \sigma^{-1}_i < \sigma^{-1}_j \le \sigma^{-1}_\ell$.
For~$(i, j) \in \boxslash(\sigma)$, we define an arc~$\alpha(i, j, \sigma)$ by
\[
\arcUp(i, j, \sigma) \eqdef \bigl( i, j, \set{k \in {]i,j[}}{\sigma^{-1}_k < \sigma^{-1}_i}, \set{k \in {]i,j[}}{\sigma^{-1}_k > \sigma^{-1}_j} \bigr).
\]
Finally, we define~$\boxslash_\equiv(\sigma)$ to be the subset of $\trianglelefteq$-minimal elements in~$\set{(i, j) \in \boxslash(\sigma)}{\arcUp(i, j, \sigma) \in \arcs_\equiv}$.
In other words, we want the flattest diagonals of~$\boxslash(\sigma)$ with~$\arcUp(i, j, \sigma) \in \arcs_\equiv$.

Symmetrically, let~$\boxbslash(\sigma) \eqdef \set{(i, j)}{1 \le i < j \le n, \, \sigma^{-1}_i > \sigma^{-1}_j \text{ and } \sigma^{-1}({]i,j[}) \cap {]\sigma^{-1}_j, \sigma^{-1}_i[} = \varnothing}$.
Order~$\boxbslash(\sigma)$ by~$(i,j) \trianglelefteq (k,\ell)$ if~$k \le i < j \le \ell$ and~$\sigma^{-1}_j \le \sigma^{-1}_\ell < \sigma^{-1}_k \le \sigma^{-1}_i$.
Associate~to each~${(i,j) \in \boxbslash(\sigma)}$ an arc~$\arcDown(i, j, \sigma) \eqdef \bigl( i, j, \set{k \in {]i,j[}}{\sigma^{-1}_k < \sigma^{-1}_j}, \set{k \in {]i,j[}}{\sigma^{-1}_k > \sigma^{-1}_i} \bigr)$.
Finally, let~$\boxbslash_\equiv(\sigma)$ be the $\trianglelefteq$-maximal elements in~$\set{(i, j) \in \boxbslash(\sigma)}{\arcDown(i, j, \sigma) \in \arcs_\equiv}$.
In other words, we want the flattest diagonals of~$\boxbslash(\sigma)$ with~$\arcDown(i, j, \sigma) \in \arcs_\equiv$.

The following statements are illustrated in \cref{fig:arcDiagramsQuotients}.

\begin{proposition}[{\cite[Sect.~2.4]{Pilaud-arcDiagramAlgebra}}]
\label{prop:insertionMap}
The noncrossing arc diagram~$\arcDiagramUp(\sigma)$ (resp.~$\arcDiagramDown(\sigma)$) contains exactly the arcs~$\arcUp(i,j,\sigma)$ for all~$(i,j) \in \boxslash_\equiv(\sigma)$ (resp.~the arcs~$\arcDown(i,j,\sigma)$ for all~$(i,j) \in \boxbslash_\equiv(\sigma)$).
\end{proposition}

\begin{corollary}
\label{coro:insertionMap}
The $\equiv$-poset~$\preccurlyeq_\sigma^\equiv$ corresponding to the $\equiv$-congruence class of a permutation~$\sigma$ has cover relations~$i \preccurlyeq_\sigma^\equiv j$ for~$(i,j) \in \boxslash_\equiv(\sigma)$ and~$j \preccurlyeq_\sigma^\equiv i$ for~$(i,j) \in \cup \boxbslash_\equiv(\sigma)$.
\end{corollary}

As already mentioned in~\cite[Sect.~2.4]{Pilaud-arcDiagramAlgebra}, the $\equiv$-poset~$\preccurlyeq_\sigma^\equiv$ can also be obtained by an insertion procedure.
We again skip the precise description as it is not needed in this paper.

\begin{example}
For the sylvester congruence~$\equiv_\textrm{sylv}$ of \cref{exm:sylvesterCongruence}, the $\equiv$-poset~$\preccurlyeq_\sigma^\equiv$ is the binary search tree obtained by insertion of the permutation~$\sigma$, as described in~\cite{HivertNovelliThibon-algebraBinarySearchTrees}.
\end{example}


\subsection{Quotient fans and quotientopes}
\label{subsec:quotientFansQuotientopes}

Fix a lattice congruence~$\equiv$ of the weak order.
Geometric realizations of the quotient~$\f{S}_n / {\equiv}$ were studied in~\cite{Reading-HopfAlgebras, PilaudSantos-quotientopes, PadrolPilaudRitter}.

\begin{theorem}[\cite{Reading-HopfAlgebras}]
The cones obtained by gluing together the maximal cones of the braid fan~$\braidFan$ corresponding to permutations in the same congruence class of~$\equiv$ define the \defn{quotient~fan}~$\quotientFan$.
\end{theorem}

Each cone of the quotient fan~$\quotientFan$ is the \defn{preposet cone} $\Cone({\preccurlyeq}) \eqdef \set{\b{x} \in \R^n}{x_i \le x_j \text{ for all } i \preccurlyeq j}$ of some preposet~$\preccurlyeq$ on~$[n]$.
Recall that a preposet~$\preccurlyeq$ (\ie a reflexive and transitive binary relation) on~$[n]$ defines an equivalence relation~${\simeq} \eqdef \bigset{\{i,j\} \in \binom{[n]}{2}}{i \preccurlyeq j \text{ and } j \preccurlyeq i}$ and a poset~${\preccurlyeq} / {\simeq}$ on the equivalence classes of~$\simeq$.
The Hasse diagram of~$\preccurlyeq$ is the Hasse diagram (\ie the oriented graph of cover relations) of the poset~${\preccurlyeq} / {\simeq}$, and we draw it as a graph where the equivalence classes of~$\simeq$ are bubbles.
Note that the dimension of~$\Cone({\preccurlyeq})$ is the number of nodes of the Hasse diagram of~$\preccurlyeq$.
Observe also that a preposet cone $\Cone({\preccurlyeq})$ is simplicial if and only if the Hasse diagram of~$\preccurlyeq$ is a forest, and in this case its faces are the preposet cones of the preposets whose Hasse diagrams are obtained by edge contraction in the Hasse diagram of~$\preccurlyeq$ (merging the two equivalence classes at the endpoints of the edge, see also \cref{def:contraction}).
We call \defn{$\equiv$-preposets} (resp.~\defn{$\equiv$-posets}) the preposets (resp.~posets) whose cone is a cone of~$\quotientFan$.
The $\equiv$-posets are precisely the posets~$\preccurlyeq^\equiv_\sigma$ described in~\cref{subsec:latticeCongruencesArcIdeals} for any choice of representatives~$\sigma$ in~$\f{S}_n/{\equiv}$.

Finally, we recall the following result of~\cite{PilaudSantos-quotientopes, PadrolPilaudRitter}.
In this paper, we will not use the precise construction, only the fact that these polytopes exist.

\begin{theorem}[\cite{PilaudSantos-quotientopes, PadrolPilaudRitter}]
The quotient fan~$\quotientFan$ is the normal fan of a \defn{quotientope}~$\Quotientope$.
\end{theorem}

\begin{proposition}[\cite{Reading-HopfAlgebras, PilaudSantos-quotientopes, PadrolPilaudRitter}]
The Hasse diagram of the quotient lattice~$\f{S}_n / {\equiv}$ is the dual graph of the quotient fan~$\quotientFan$, or the graph of the quotientope~$\Quotientope$, oriented in the linear direction~$\b{\alpha} \eqdef \sum_{i \in [n]} (2i-n-1) \, \b{e}_i$.
\end{proposition}

\begin{example}
\label{exm:sylvesterFanAssociahedron}
\begin{figure}
	\capstart
	\centerline{\includegraphics[scale=.48]{TamariLattice4} \; \includegraphics[scale=.6]{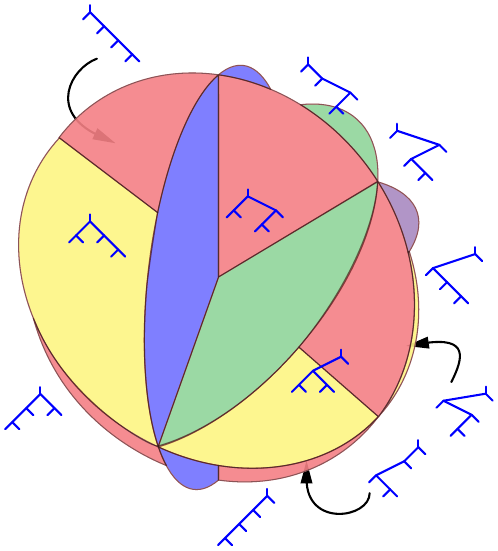} \hspace{-.3cm} \includegraphics[scale=.6]{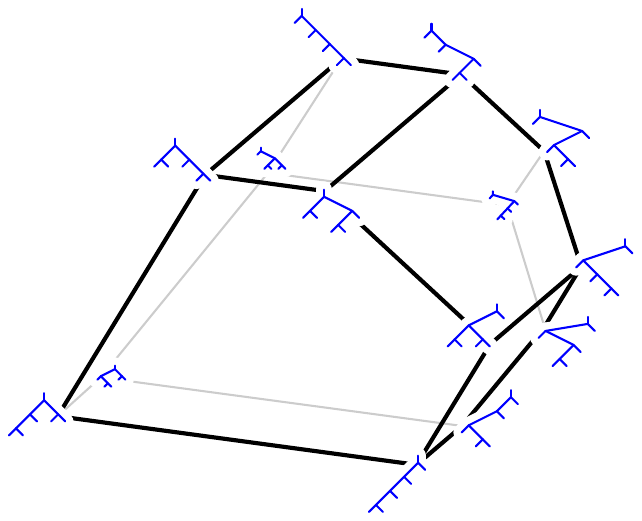}}
	\caption{The cover graph of the Tamari lattice (left) can be seen as the dual graph of the sylvester fan~$\quotientFan[\textrm{sylv}]$ (middle) or as the graph of J.-L.~Loday's \mbox{associahedron (right).~\cite[Fig.~5]{PadrolPilaudRitter}}}
	\label{fig:Tamari4}
\end{figure}
For the sylvester congruence~$\equiv_\textrm{sylv}$ of \cref{exm:sylvesterCongruence}, the quotient fan~$\quotientFan[\textrm{sylv}]$ has a cone~$\Cone(S) \eqdef \set{\b{x} \in \hyp}{\b{x}_a \le \b{x}_b \text{ if there is a directed path from } a \text{ to } b \text{ in } S}$ for each Schröder tree~$S$ with $n+1$ leaves.
In particular, $\quotientFan[\textrm{sylv}]$ has a chamber for each binary tree with $n+1$ leaves (or equivalently $n$ nodes).
See \cref{fig:Tamari4}\,(middle).
The quotient fan~$\quotientFan[\textrm{sylv}]$ is the normal fan of the classical associahedron~$\Asso$ of~\cite{ShniderSternberg, Loday}, defined as the convex hull of the points~$\sum_{j \in [n]} \ell(T,j) \, r(T,j) \, \b{e}_j$ for all binary trees~$T$ on~$n$ nodes, where $\ell(T,j)$ and~$r(T,j)$ respectively denote the numbers of leaves in the left and right subtrees of the node~$j$ of~$T$ (labeled in inorder).
See \cref{fig:Tamari4}\,(right).
\end{example}


\subsection{Relevant congruences}
\label{subsec:relevantCongruences}

\enlargethispage{.5cm}
We briefly gather some relevant congruences of the weak order:

\begin{enumerate}
\item \label{item:trivialCongruence}
The \defn{trivial congruence}~$\equiv_\textrm{triv}$ is defined by the ideal~$\arcs_n$ of all arcs.
It has a congruence class for each permutation.
The quotient~$\f{S}_n/{\equiv_\textrm{triv}}$ is the weak order, the quotient fan~$\quotientFan[\textrm{triv}]$ is the braid fan, and the quotientope~$\Quotientope[\textrm{triv}]$ is the permutahedron~$\Perm$.

\item \label{item:sylvesterCongruence}
The \defn{sylvester congruence} is defined by the ideal~$\arcs_\textrm{sylv} \eqdef \set{(i, j, {]i,j[}, \varnothing)}{1 \le i < j \le n}$ of up arcs.
It has a congruence class for each binary tree, given by the linear extensions of this tree (when labeled in inorder and oriented towards its root)~\cite{Tonks, HivertNovelliThibon-algebraBinarySearchTrees}.
The quotient~$\f{S}_n/{\equiv_\textrm{sylv}}$ is the Tamari lattice, the quotient fan~$\quotientFan[\textrm{sylv}]$ is the sylvester fan, and the quotientope~$\Quotientope[\textrm{sylv}]$ is the associahedron~$\Asso$ of~\cite{ShniderSternberg, Loday}.

\item \label{item:recoilCongruence}
The \defn{recoil congruence}~$\equiv_\textrm{rec}$ is defined by the ideal~$\arcs_\textrm{rec} = \set{(i, i+1, \varnothing, \varnothing)}{i \in [n-1]}$ of basic arcs.
It has a congruence class for each subset~$I \subseteq [n-1]$ given by the permutations whose recoils (descents of the inverse) are at positions in~$I$.
The quotient~$\f{S}_n/{\equiv_\textrm{rec}}$ is the boolean lattice, the quotient fan~$\quotientFan[\textrm{rec}]$ is defined by the hyperplanes~$x_i = x_{i+1}$ for~$i \in [n-1]$, and the quotientope~$\Quotientope[\textrm{rec}]$ is the parallelotope~$\sum_{i \in [n-1]} [\b{e}_i, \b{e}_{i+1}]$.

\item \label{item:CambrianCongruence}
For an arc~$\arc \eqdef (i, j, A, B) \in \arcs_n$, the \defn{$\arc$-Cambrian congruence}~$\equiv_\arc$ is defined by the ideal~$\arcs_\arc$ of the arc poset generated by~$\arc$.
It has a congruence class for each  $\arc$-Cambrian tree~\cite{LangePilaud, ChatelPilaud, PilaudPons-permutrees}.
The quotient~$\f{S}_n / {\equiv_\arc}$ is the \defn{$\arc$-Cambrian lattice}~\cite{Reading-CambrianLattices}, the quotient fan~$\quotientFan[\arc]$ is the \defn{$\arc$-Cambrian fan} of~\cite{ReadingSpeyer} and the quotientope~$\Quotientope[\arc]$ is the \defn{$\arc$-associahedron} of~\cite{HohlwegLange}.

\item \label{item:permutreeCongruence}
For~$\decoration \in \Decorations^n$, the \defn{$\decoration$-permutree congruence}~$\equiv_\decoration$ is defined by the ideal~$\arcs_\decoration$ of arcs which do not pass above the points~$j$ with~$\decoration_j \in \{\upCirc, \upDownCirc\}$ nor below the points~$j$ with~${\decoration_j \in \{\downCirc, \upDownCirc\}}$.
Its congruence classes are the sets of linear extensions of the \linebreak $\decoration$-permutrees~\cite{PilaudPons-permutrees}.
The quotient~$\f{S}_n / {\equiv_\decoration}$ is the \defn{$\decoration$-permutree lattice}~\cite{PilaudPons-permutrees}, the quotient fan~$\quotientFan[\decoration]$ is the \defn{$\decoration$-permutree fan} of~\cite{PilaudPons-permutrees} and the quotientope~$\Quotientope[\decoration]$ is the \defn{$\decoration$-permutreehedron} of~\cite{PilaudPons-permutrees}.

\item \label{item:abSashesCongruence}
For~$a, b \ge 1$ define the \defn{$(a,b)$-sashes congruence}~$\equiv_{\textrm{sh}(a,b)}$ by the ideal~$\arcs_{\textrm{sh}(a,b)}$ of arcs which do not pass above $a$ consecutive points, nor below $b$ consecutive points.
Note that~${{\equiv_{\textrm{sh}(n,n)}} = {\equiv_{\textrm{triv}}}}$, that~${\equiv_{\textrm{sh}(1,n)}} = {\equiv_{\textrm{sylv}}}$, that~${\equiv_{\textrm{sh}(1,1)}} = {\equiv_{\textrm{rec}}}$.
The $(1,2)$-sashes congruence was studied in~\cite{Law} and its equivalence classes are counted by the Pell numbers.
The $(2,2)$-sashes congruence was studied in~\cite{ChatelPilaud} as a particular Baxter-Cambrian congruence, and its equivalence classes are counted by central binomial coefficients.
The ${(1,n-1)}$-sashes congruence is studied in~\cite{LaniniNovelli}, and the generating function of its number of equivalence classes is the inverse of the truncation of~${1 + x/(1-x-\sum_{i=0}^{n-1} C_i x^{i+1})}$, where~$C_i \eqdef \frac{1}{i+1} \binom{2i}{i}$ is the $i$th Catalan number.

\item \label{item:BaxterCongruence}
The \defn{diagonal rectangulation congruence}~$\equiv_\textrm{DR}$ (or \defn{Baxter congruence}) is defined by the ideal of arcs that do not cross the horizontal axis, \ie~$\arcs_\textrm{DR} \!=\! \set{(i, j, A, B) \in \arcs_n}{A \!=\! \varnothing \text{ or } B \!=\! \varnothing}$.
Its congruence classes correspond to diagonal rectangulations~\cite{LawReading} or equivalently pairs of twin binary trees~\cite{Giraudo}, which are counted by the Baxter numbers.
The quotient $\f{S}_n / {\equiv_\textrm{DR}}$ is the \defn{diagonal rectangulation lattice}~\cite{LawReading}, the quotientope~$\Quotientope[\textrm{DR}]$ is the \defn{weak rectangulotope} defined as the Minkowski sum of two opposite Loday associahedra~\cite{LawReading,CardinalPilaud}.

\item \label{item:GenericRectangulationCongruence}
The \defn{generic rectangulation congruence}~$\equiv_\textrm{GR}$ is defined by the ideal of arcs that cross the horizontal axis at most once, \ie~$\arcs_\textrm{GR} = \set{(i, j, A, B) \in \arcs_n}{A \text{ and } B \text{ are both intervals}}$.
Its congruence classes correspond to generic rectangulations~\cite{Reading-rectangulations}, \ie any rectangulation of the square with no $4$ incident rectangles, considered equivalent under wall slides.
The quotient~$\f{S}_n / {\equiv_\textrm{GR}}$ is the \defn{generic rectangulation lattice}~\cite{Reading-rectangulations}, and the quotientope~$\Quotientope[\textrm{GR}]$  is the strong rectangulotope studied in~\cite{CardinalPilaud}.

\item 
For~$p \ge 1$, the \defn{$p$-recoil congruence}~$\equiv_{p\textrm{-rec}}$ is defined by the ideal of arcs of length at most~$p$, \ie~$\arcs_{p\textrm{-rec}} = \set{(a, b, A, B) \in \arcs_n}{b-a \le p}$.
Its congruence classes correspond to acyclic orientations of the graph on~$[n]$ with edges~$(a,b)$ for~$|a-b| \le p$.
See \cite{NovelliReutenauerThibon, Reading-HopfAlgebras, Pilaud-brickAlgebra}.
The quotientope~$\Quotientope[p\textrm{-rec}]$ is the graphical zonotope of the graph on~$[n]$ with edges~$\{i,j\}$ for~$|i-j| \le p$.

\item
For~$p \ge 1$, the \defn{$p$-twist congruence}~$\equiv_{p\textrm{-twist}}$ is defined by the ideal of arcs passing below at most~$p$ points, \ie~$\arcs_{p\textrm{-twist}} = \set{(a, b, A, B) \in \arcs_n}{|B| \le p}$.
Its congruence classes correspond to certain acyclic pipe dreams~\cite{Pilaud-brickAlgebra}.
The quotientope~$\Quotientope[p\textrm{-twist}]$ is a brick polytope of~\cite{PilaudSantos-brickPolytope,PilaudStump-brickPolytope}.

\item
The graph associahedra congruences of~\cite{BarnardMcConville} were defined by the normal fan of the associahedra of filled graphs. Equivalently they are the congruences~$\equiv$ whose subarc minimal non arcs are up arcs.
\end{enumerate}


\subsection{Simple congruences}
\label{subsec:simpleCongruences}

For any~$1 \le i < j \le n$, we always denote by~$u(i,j) \eqdef (i, j, {]i,j[}, \varnothing)$ (resp.~by~$d(i,j) \eqdef (i, j, \varnothing, {]i,j[})$) the \defn{up} (resp.~\defn{down}) \defn{arc} joining~$i$ to~$j$.
We say that $\arc$ is a \defn{subarc minimal arc} in a collection of arcs~$\arcs$ if no strict subarc of~$\arc$ belongs to~$\arcs$.
This paper focuses on the following congruences of the weak order.

\begin{definition}[{\cite[Sect.~4.4]{HoangMutze}}]
\label{def:simpleCongruence}
A lattice congruence~$\equiv$ of the weak order is \defn{simple} if all subarc minimal arcs in~$\arcs_n \ssm \arcs_\equiv$ are up arcs or down arcs.
\end{definition}

\begin{example}
Among the examples of \cref{subsec:relevantCongruences},
\begin{itemize}
\item the trivial congruence,
\item the sylvester congruence~\cite{Tonks, HivertNovelliThibon-algebraBinarySearchTrees},
\item the recoil congruence,
\item the Cambrian congruences~\cite{Reading-CambrianLattices, ChatelPilaud},
\item the permutree congruences~\cite{PilaudPons-permutrees},
\item the $(a,b)$-sashes congruences~\cite{Law,LaniniNovelli},
\item the graph associahedra congruences of~\cite{BarnardMcConville},
\end{itemize}
are all simple congruences of the weak order.
\end{example}

\begin{remark}
\label{rem:permutreeCongruences}
Note that the permutree congruences of~\cite{PilaudPons-permutrees} are by definition the simple congruences for which all subarc minimal arcs in~$\arcs_n \ssm \arcs_\equiv$ have length~$2$, meaning that they are of the form~$u(i-1, i+1)$ or~$u(i-1, i+1)$ for some~$2 \le i \le n-1$.
\end{remark}

\begin{example}
In contrast, among the examples of \cref{subsec:relevantCongruences},
\begin{itemize}
\item the diagonal rectangulation congruence~\cite{LawReading,Giraudo,CardinalPilaud},
\item the generic rectangulation congruence~\cite{Reading-rectangulations,CardinalPilaud},
\item the $p$-recoil congruence~\cite{NovelliReutenauerThibon, Reading-HopfAlgebras, Pilaud-brickAlgebra},
\item the $p$-twist congruence~\cite{Pilaud-brickAlgebra},
\end{itemize}
are non-simple congruences of the weak order (except when~$n \le 3$ or~$p = 1$).
\end{example}

As all congruences, the simple congruences of the weak order have a natural lattice structure.

\begin{proposition}
\label{prop:simpleCongruenceMeetSemilattice}
The set of simple congruences induces a meet subsemilattice of the congruence lattice of the weak order.
Hence, they form a lattice under refinement.
\end{proposition}

\begin{proof}
For two congruences~$\equiv$ and~$\equiv'$, the non arcs of~${\equiv} \meet {\equiv'}$ are
\[
\arcs_n \ssm \arcs_{{\equiv} \meet {\equiv'}} = \arcs_n \ssm (\arcs_{\equiv} \cap \arcs_{\equiv'}) = (\arcs_n \ssm \arcs_{\equiv}) \cup (\arcs_n \ssm \arcs_{\equiv'})
\]
so that any subarc minimal non-arc of~${\equiv} \meet {\equiv'}$ is a subarc minimal non-arc of~${\equiv}$ or a subarc minimal non-arc of~${\equiv'}$.
Thus, the simple congruences form a finite meet semilattice, hence a lattice under refinement.
\end{proof}

\begin{remark}
Note that the simple congruences do not induce a sublattice of the congruence lattice.
For instance, the Baxter congruence is not simple, while it is the join (in the congruence lattice) of the sylvester and anti-sylvester congruence (where the allowed arcs are exactly the down arcs), which are both simple congruences.
\end{remark}

\begin{remark}
As already observed in~\cite[Sect.~4.4]{HoangMutze}, the number of simple essential congruences of the weak order on~$\f{S}_n$ is~$C_{n-1}^2$, where~$C_n \eqdef \frac{1}{n+1} \binom{2n}{n}$ is the $n$th Catalan number.
Indeed, the subarc minimal arcs in~$\arcs_n \ssm \arcs_\equiv$ form two (above and below) independent pairwise nonnesting collections of intervals.
The number of all (non-necessarily essential) simple congruences of the weak order on~$\f{S}_n$ is~$\sum_{k \ge 1} \sum_{0 = i_0 < i_1 < \dots < i_k = n} \prod_{i \in [k]} C_{i_j - i_{j-1}}$, and the first few numbers are
\[
1, 2, 7, 38, 274, 2350, 22531, 233292, 2555658, 29232554, 346013450, 4211121946, 52446977292, \dots
\]
\end{remark}

One of the contributions of the present paper is a simple combinatorial proof of the following statement.
This statement was originally proved by H.~Hoang and T.~Mütze in~\cite[Sect.~4.4]{HoangMutze} with a purely combinatorial (but slightly involved) proof.
In their paper on the lattice theory of torsion classes, L.~Demonet, O.~Iyama, N.~Reading, I.~Reiten, and H.~Thomas provided an alternative proof~\cite[Sect.~6.3]{DemonetIyamaReadingReitenThomas} in terms of quiver representation theory.
In this paper, we propose an alternative simple combinatorial proof based on the description of the poset insertion map of \cref{subsec:insertionMap}.

\begin{theorem}[{\cite[Sect.~4.4]{HoangMutze} \& \cite[Sect.~6.3]{DemonetIyamaReadingReitenThomas}}]
\label{thm:simpleCongruences}
The following assertions are equivalent for an essential lattice congruence~$\equiv$ of the weak order:
\begin{enumerate}[(i)]
\item $\equiv$ is a simple congruence,
\item the Hasse diagram of any $\equiv$-poset is a tree,
\item the Hasse diagram of the lattice quotient~$\f{S}_n/{\equiv}$ is regular,
\item the quotient fan~$\quotientFan$ is a simplicial fan,
\item the quotientope~$\Quotientope$ is a simple polytope.
\end{enumerate}
\end{theorem}

\begin{proof}
The equivalences (ii) $\iff$ (iii) $\iff$ (iv) $\iff$ (v) are standard.

(i)~$\Longrightarrow$~(ii): See \cref{subsec:posetsSimpleCongruences}.

(ii)~$\Longrightarrow$~(i): Assume that~$\equiv$ has a subarc minimal non arc~$\alpha \eqdef (i, j, A, B)$ with~$A \ne \varnothing \ne B$, let~$a \eqdef \max(A)$ and~$b \eqdef \min(B)$, and let~$\sigma$ be the permutation with arc diagram~$\arcDiagramDown[](\sigma) = \{\alpha\}$. Then the four arcs
\[
(i, a, A \ssm \{a\}, {]i,a[} \cap B)
\qquad
(a, j, \varnothing, {]a,j[} \cap B)
\qquad
(i, b, {]i,b[} \cap A, \varnothing)
\qquad\text{and}\qquad
(b, j, {]b,j[} \cap A, B \ssm \{b\})
\]
all belong to~$\arcs_\equiv$ (otherwise, $\alpha$ would not be a subarc minimal non arc).
It thus immediately follows from~\cref{prop:insertionMap} that the preposet~$\preccurlyeq_\sigma^\equiv$ has cover relations~$a \preccurlyeq_\sigma^\equiv i$, $a \preccurlyeq_\sigma^\equiv j$, $i \preccurlyeq_\sigma^\equiv b$, and $j \preccurlyeq_\sigma^\equiv b$.
Hence, the Hasse diagram of the $\equiv$-poset~$\preccurlyeq_\sigma^\equiv$ is not a tree.
\end{proof}

\begin{remark}
If we considered the non-essential congruences, the $\equiv$-posets would be forests instead of trees.
As the non-essential quotients of the weak order can be understood as Cartesian products of essential quotients of weak orders of smaller rank, we focus on essential congruences in \cref{thm:simpleCongruences} and throughout the paper.
\end{remark}

\begin{remark}
\enlargethispage{.3cm}
The arcs which are neither up nor down arcs correspond to doubly join irreducibles of the weak order on permutations.
One can naturally ask whether a statement similar to \cref{thm:simpleCongruences} could hold for other finite Coxeter groups.
The congruences with a simplicial quotient fan are indeed generated by contracting doubly join irreducibles~\cite[Thm.~1.13]{DemonetIyamaReadingReitenThomas}, but
\begin{itemize}
\item in type~$D_4$, there is a doubly join irreducible whose contraction yields a congruence whose quotient fan is not simplicial~\cite[Exm.~6.4]{DemonetIyamaReadingReitenThomas} (in contrast to \cref{thm:simpleCongruences}),
\item in type~$B_3$, the meet of two simple congruences is not necessarily simple (in contrast to \cref{prop:simpleCongruenceMeetSemilattice}).
\end{itemize}
\end{remark}


\section{Separating trees and posets of simple congruences}
\label{sec:separatingTreesPosetsSimpleCongruences}

In this section, we define separating trees (\cref{subsec:separatingTrees}) and prove that they provide a combinatorial description of the vertices of the simple quotientopes (\cref{subsec:posetsSimpleCongruences}).


\subsection{Separating trees}
\label{subsec:separatingTrees}

For a node~$j$ in a directed tree~$T$, we call \defn{parents} (resp.~\defn{children}) of~$j$ the targets of the outgoing arcs (resp.~the  sources of the incoming arcs) at~$j$ and \defn{ancestor} (resp.~\defn{descendant}) \defn{subtrees} of~$j$ the subtrees attached to them.
This paper focuses on the following family of trees.

\begin{definition}
\label{def:separatingTree}
A \defn{separating tree} is a directed tree on~$[n]$ such that
\begin{itemize}
\item each node has at most two parents and at most two children,
\item a node~$j$ with two parents (resp.~children) \defn{separates} its two ancestor (resp.~descendant) subtrees, meaning that all nodes in the left subtree~$L$ are smaller than~$j$ while all nodes in the right subtree~$R$ are larger than~$j$ (and we then write~$L < j < R$ and~$L < R$).
\end{itemize}
\end{definition}

\begin{example}
\label{exm:separatingTrees}
\cref{fig:separatingTrees} represents some separating trees. We use the following conventions:
\begin{itemize}
\item all edges are oriented from bottom to top, so that we omit to add arrows to specify the direction of the edges,
\item the vertices are ordered increasingly from left to right.
\end{itemize}
\begin{figure}[h]
	\capstart
	\centerline{\includegraphics[scale=.85]{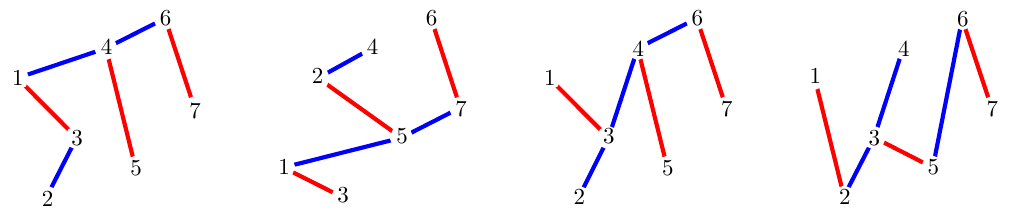}}
	\caption{Some separating trees on~$[7]$.}
	\label{fig:separatingTrees}
\end{figure}
\end{example}

\begin{remark}
\label{rem:separatingTreesVSPermutrees}
Readers familiar with~\cite{PilaudPons-permutrees} might wonder what the difference between permutrees and separating trees is.
The subtle distinction is that empty ancestor or descendant subtrees are allowed in permutrees and not in separating trees.
For instance, in a separating tree, it is impossible to say that a node has two children, one of which is empty.
We will see in \cref{rem:separatingTreesArePermutrees} that any separating tree appears in at least one permutree congruence.
However, simple congruences are much richer than permutree congruences.
\end{remark}

\begin{remark}
The number of separating trees with~$n$ nodes is given by the following table:
\[
\begin{array}{c|ccccccccc}
n & 1 & 2 & 3 & 4 & 5 & 6 & 7 & 8 & \dots \\
\hline
\#\text{ separating trees} & 1 & 2 & 8 & 42 & 264 & 1898 & 15266 & 135668 & \dots
\end{array}
\]
\end{remark}

The end of this section gathers some structural lemmas about separating trees.
We start with two lemmas ensuring the existence of certain paths in a separating tree.

\begin{lemma}
\label{lem:separatingTree1}
Consider three nodes~$1 \le i < j < k \le n$ in a separating tree~$T$. Then
\begin{itemize}
\item if~$i \leftarrow k$ is an edge of~$T$, then there is a directed path in~$T$ joining either~$i$ to~$j$, or~$j$ to~$k$,
\item if~$i \to k$ is an edge of~$T$, then there is a directed path in~$T$ joining either~$j$ to~$i$, or~$k$ to~$j$.
\end{itemize}
\end{lemma}

\begin{proof}
Assume that~$1 \le i < j < k \le n$ and~$i \leftarrow k$ is an edge of~$T$.
As~$T$ is a tree, there is a (\apriori undirected) path~$\pi$ joining~$j$ and the edge~$i \leftarrow k$.
If~$\pi$ is not a directed path, then there exists some node~$\ell$ along~$\pi$ such that~$\pi$ has either two incoming or two outgoing edges at~$\ell$.
This contradicts the separating property, as $i < j < k$ so that~$\ell$ cannot separate~$j$ from~$\{i,k\}$.
Hence, the path~$\pi$ is directed, either from~$i \leftarrow k$ to~$j$, or from~$j$ to~$i \leftarrow k$.
Again using~$i < j < k$, the path~$\pi$ cannot be incoming at~$i$, nor outgoing from~$k$.
We conclude that $T$ contains a directed path either from~$i$ to~$j$, or from~$j$ to~$k$.
The second part of the statement is symmetric.
\end{proof}

\begin{lemma}
\label{lem:separatingTree3}
Consider a node~$j$ in a separating tree~$T$ with two ancestor (resp.~descendant) subtrees~$L$ and~$R$ with~$L < R$, and let~$M \eqdef \max(L)$ and~$m \eqdef \min(R)$. Then for each~${M < k < m}$, there is a directed path in~$T$ from~$k$ to~$j$ (resp.~from~$j$ to~$k$).
\end{lemma}

\begin{proof}
Assume that~$j$ has two ancestors, the proof for two descendants being symmetric.
As~$T$ is a tree, there is a (\apriori undirected) path~$\pi$ between~$k$ and~$j$.
As~$M < k < m$, we obtain that~$k$ is not an ancestor of~$j$, so that the path~$\pi$ ends with an incoming edge at~$j$.
If~$\pi$ is not a directed path, then there exists some node~$\ell$ along~$\pi$ such that~$\pi$ has either two outgoing or two incoming edges at~$\ell$.
This contradicts the separating property, as $M < k < m$ so that~$\ell$ cannot separate~$k$ from~$\{M,m\}$.
We conclude that~$\pi$ is a directed path from~$k$ to~$j$.
\end{proof}

\cref{lem:separatingTree1,lem:separatingTree3,subsec:insertionMap} motivate the following definitions of arcs associated to a separating tree~$T$, which will be crucial in \cref{def:admissibleSeparatingTrees,prop:admissibleSeparatingTrees}.

\begin{definition}
\label{def:mandatoryArcsSeparatingTree}
The \defn{mandatory arcs} of a separating tree~$T$ are the arcs of the form~$(i, k, A, B)$, where~$i < k$ form an edge of~$T$ and~$A$ (resp.~$B$) is the set of nodes~$j \in {]i,k[}$ such that there is a directed path in~$T$ joining the node $j$ to the edge~$i-k$ (resp.~the edge~$i-k$ to the node~$j$).
\end{definition}

\begin{definition}
\label{def:forbiddenArcsSeparatingTree}
The \defn{forbidden up} (resp.~\defn{down}) \defn{arcs} of a separating tree~$T$ are the up arcs~$u(M, m)$ (resp.~down arcs~$d(M, m)$), with~$M \eqdef \max(L)$ and~$m \eqdef \min(R)$ where~$L$ and~$R$ are two ancestor (resp.~descendant) subtrees of a node~$j$ of~$T$ with~$L < R$.
\end{definition}

\begin{example}
For instance, \cref{fig:mandatoryForbiddenArcsSeparatingTrees} shows the mandatory (blue, top) and forbidden (red, bottom) arcs of the four separating trees of \cref{fig:separatingTrees}.

\begin{figure}
	\capstart
	\centerline{\includegraphics[scale=.85]{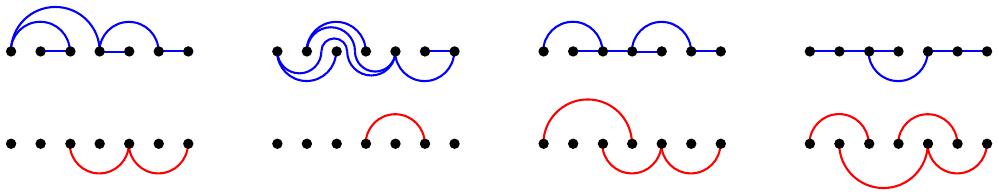}}
	\caption{The mandatory arcs (blue, top) and the forbidden arcs (red, bottom) of the four separating trees of \cref{fig:separatingTrees}.}
	\label{fig:mandatoryForbiddenArcsSeparatingTrees}
\end{figure}
\end{example}

\begin{remark}
\label{rem:separatingTreeFallHorizontalAxis}
If we draw $T$ such that its nodes are ordered increasingly from left to right, and its edges are oriented from bottom to top, then the mandatory arcs of~$T$ are obtained by letting $T$ fall down to the horizontal axis, allowing the edges to curve, but not to cross each other nor to pass through any node.
\end{remark}

Finally, we give a planar characterization of separating trees which will be useful in the proof of \cref{prop:simpleImpliesSeparatingTrees}.

\begin{lemma}
\label{lem:separatingTree2}
Let~$G$ be a directed graph on~$[n]$ drawn such that the nodes are ordered increasingly from left to right, and the edges are oriented from bottom to top.
Then~$G$ is a separating tree~if~and~only~if
\begin{enumerate}[(i)]
\item each node~$j$ of~$G$ has at most one outgoing (resp.~incoming) neighbor~$i < j$ and at most one outgoing (resp.~incoming) neighbor~$k > j$,
\item if~$i < j < k$ and~$j$ has two parents (resp.~children) and~$i-k$ is an edge of~$G$, then~$i-k$ passes below (resp.~above)~$j$.
\end{enumerate}
In other words, when we draw a wall above each node with two parents and below each node with two children, the edges of the separating tree never cross the walls.
\end{lemma}

\begin{proof}
First, if~$G$ is a separating tree, then it satisfies~(i) by definition and~(ii) by a direct application of \cref{lem:separatingTree1}.
Conversely, assume that~$G$ satisfies~(i) and~(ii).
Assume that~$G$ contains a (non oriented) cycle~$C$, and let~$j$ be the topmost node of~$C$.
Hence~$j$ has two incoming edges~$p \to j$ and~$q \to j$, and we can assume that~$p < j < q$ by Condition~(i).
Therefore, $C$ has an edge~$i-k$ with~$i < j < k$ to connect~$p$ with~$q$.
As~$j$ is the topmost node of~$C$, the edge~$i-k$ passes below~$j$, contradicting Condition~(ii).
We conclude that~$G$ is a tree.
Consider now a node~$j$ of~$G$.
By~(i), $j$ has at most two parents and at most two children.
We now prove by induction on~$j$ that if~$j$ has two parents (resp.~children), then all nodes in its left ancestor (resp.~descendant) subtree are smaller than~$j$.
Assume by contradiction that this subtree contains some~$\ell > j$, and consider the path~$\pi$ from~$j$ to~$\ell$ in~$G$.
Along this path, there is an edge~$i - k$ with~$i < j < k$.
By~(ii), this edge must pass below (resp.~above)~$j$.
Therefore, the path~$\pi$ is not a directed path from~$j$ to~$k$.
Hence, there exists a node~$h < j$ along~$\pi$ such that~$\pi$ has two incoming (resp.~outgoing) arrows at~$h$, and $\ell$ belongs to the left descendant (resp.~ancestor) subtree of~$h$.
By induction, we obtain that~$\ell < h < j$, contradicting our assumption that~$\ell > j$.
The proof for the right ancestor (resp.~descendant) is symmetric.
We conclude that~$G$ is indeed a separating tree.
\end{proof}


\subsection{Posets of simple congruences}
\label{subsec:posetsSimpleCongruences}

We now provide a short combinatorial proof of the implication~(i)~$\Longrightarrow$~(ii) of \cref{thm:simpleCongruences}.
Our proof is a very elementary approach based on the insertion map described in~\cref{subsec:insertionMap}.
In fact, we prove the following stronger statement on the $\equiv$-posets of a simple congruence~$\equiv$.

\begin{proposition}
\label{prop:simpleImpliesSeparatingTrees}
If~$\equiv$ is a simple essential congruence of the weak order, then the Hasse diagram of any $\equiv$-poset is a separating tree.
\end{proposition}

\begin{proof}
Assume that $\equiv$ is simple, and consider the Hasse diagram~$H$ of a $\equiv$-poset~$\preccurlyeq$.
Pick a linear extension~$\sigma$ of~$\preccurlyeq$.
By \cref{coro:insertionMap}, we obtain~$H$ starting from the permutation table of~$\sigma$.
We check that~$H$ fulfills the two conditions of \cref{lem:separatingTree2}.

First, for each $j \in [n]$, there is at most one~$i \in [n]$ such that~$(i,j) \in \boxslash_\equiv(\sigma)$, so that $j$ has at most one incoming neighbor~$i < j$ in~$H$.
Similarly, each $j \in [n]$ has at most one incoming neighbor~$k > j$.
The same arguments holds for outgoing neighbors
(Note that this part holds for all congruences, not only the simple ones.)

Second, assume that~$i < j < k$ are such that~$j$ has two parents~$p < j < q$, and~$i$ and~$k$ appear after~$j$ in~$\sigma$.
Assume for instance that~$p$ appears before~$q$ in~$\sigma$ (the other case is symmetric).
Consider the rightmost point~$r$ of the table of~$\sigma$ in the rectangle~${]p, q]} \times {]\sigma^{-1}_p,\sigma^{-1}_q]}$.
Note that~${p \le r < j}$ as~$(j,q) \in \boxslash(\sigma)$, and that~$(r,q) \in \boxslash(\sigma)$.
As~$(r,q) \trianglelefteq (j, q)$ and~$(j,q) \in \boxslash_\equiv(\sigma)$ we therefore obtain that~$\arcUp(r, q, \sigma) \notin \arcs_\equiv$.
Let~$\alpha \eqdef (u, v, U, V)$ be a subarc minimal arc of~$\arcs_n \ssm \arcs_\equiv$ which is a subarc of~$\arcUp(r, q, \sigma)$.
As~$\arcDown(p, j, \sigma)$ and~$\arcUp(j, q, \sigma)$ both belong to~$\arcs_\equiv$, we have~$p \le u < j < v \le r$.
As~$\equiv$ is a simple congruence, we have either~$U = \varnothing$ or~$V = \varnothing$.
If~$U = \varnothing$, then~$V = {]u,v[} \ni j$, contradicting that~$\alpha$ is a subarc of~$\arcUp(r, q, \sigma)$.
Hence, $U = {]u,v[}$ and~$V = \varnothing$, so that all points of~$U$ appear before~$j$ in~$\sigma$, hence before~$i$ and~$k$.
This implies that~$\alpha = (u, v, {]u,v[}, \varnothing)$ is a subarc of~$\arcUp(i, k, \sigma)$ and~$\arcDown(i, k, \sigma)$.
Hence,~$i-k$ cannot be an edge of~$H$.
\end{proof}

We can actually characterize the $\equiv$-posets of a simple congruence~$\equiv$.
Recall that we have defined the mandatory and forbidden arcs of a separating tree in \cref{def:mandatoryArcsSeparatingTree,def:forbiddenArcsSeparatingTree}.

\begin{definition}
\label{def:admissibleSeparatingTrees}
Fix an essential lattice congruence~$\equiv$ of the weak order.
A separating tree~$T$ is \defn{$\equiv$-admissible} if~$\arcs_\equiv$ contains all mandatory arcs of~$T$ and none of the forbidden arcs of~$T$.
\end{definition}

\begin{example}
For instance, the separating trees which are admissible for the sylvester congruence (resp.~$\decoration$-permutree congruence) are precisely the binary trees (resp.~$\decoration$-permutrees) on~$[n]$ to which we have deleted all pending leaves to keep only the~$n$ internal nodes and~$n-1$ internal edges.
\end{example}

\begin{lemma}
\label{lem:admissibleSeparatingTreesWithCommonLinearExtensions}
If~$T$ and~$T'$ are two $\equiv$-admissible separating trees with a common linear extension, then~${T = T'}$.
\end{lemma}

\begin{proof}
\enlargethispage{.3cm}
Assume that~$T$ and~$T'$ are two distinct $\equiv$-admissible separating trees with a common linear extension. 
By symmetry, we can assume that there are~$1 \le p < q \le n$ such that the edge from~$p$ to~$q$ is in~$T$ but not in the transitive closure of~$T'$.
As $T$ and~$T'$ have a common linear extension, this implies that~$p$ and~$q$ are incomparable in~$T'$.
As~$T'$ is a tree, there is an undirected path~$\pi$ from~$p$ to~$q$ in~$T'$.
Let~$j$ be the last node along~$\pi$ with either two outgoing edges, or two incoming edges.
By symmetry, we can assume that~$\pi$ has two outgoing edges at~$j$.
Let~$M$ (resp.~$m$) denotes the maximum (resp.~minimum) of the left (resp.~right) ancestor subtree of~$j$ in~$T'$.
Note that~$p \le M < j < m \le q$ since~$T'$ is separating.
For any~$k \in {]M,m[}$, \cref{lem:separatingTree3} ensures the existence of a directed path from~$k$ to~$j$ in~$T'$, hence from~$k$ to~$q$.
Consider now the edge between~$p$ and~$q$ in~$T$, and let~$(p, q, A, B)$ denote the mandatory arc of~$T$ defined in \cref{def:mandatoryArcsSeparatingTree}.
As~$T$ and~$T'$ are both $\equiv$-admissible, we obtain that~$\arcs_\equiv$ contains the mandatory arc~$(p, q, A, B)$ of~$T$ but not the forbidden arc~$(M, m, {]M,m[}, \varnothing)$ of~$T'$.
As~$p \le M < m \le q$, this implies that there is~$k \in {]M,m[} \ssm A$.
By \cref{lem:separatingTree1,def:mandatoryArcsSeparatingTree}, there is a directed path from~$q$ to~$k$ in~$T$.
We conclude that there is a path from~$q$ to~$k$ in~$T$, and from~$k$ to~$q$ in~$T'$, so that~$T$ and~$T'$ cannot have a common linear extension, contradicting our original assumption.
\end{proof}

\begin{proposition}
\label{prop:admissibleSeparatingTrees}
If~$\equiv$ is a simple essential congruence of the weak order, then the Hasse diagrams of the $\equiv$-posets are precisely the $\equiv$-admissible separating trees.
\end{proposition}

\begin{proof}
Consider first a $\equiv$-poset.
By \cref{prop:simpleImpliesSeparatingTrees}, its Hasse diagram is a separating tree~$T$, and we just need to prove that~$T$ is $\equiv$-admissible.
First, $\arcs_\equiv$ clearly contains the mandatory arcs of~$T$ by \cref{subsec:insertionMap}.
Consider now a forbidden up arc~$u(M,m)$ of~$T$.
Since~$M$ and~$m$ are incomparable in~$T$, there exists a linear extension~$\sigma$ of~$T$ in which~$M$ appears just before~$m$.
Hence,~$(M,m)$ is in~$\boxslash_\sigma$ and is minimal for~$\trianglelefteq$.
Moreover, \cref{lem:separatingTree3} implies that~${\arcUp(M, m, \sigma) = u(M,m)}$.
As~$(M,m)$ is not an edge of~$T$, we conclude that~${\arcUp(M, m, \sigma) = u(M,m) \notin \arcs_\equiv}$.
The proof is symmetric for forbidden down arcs.
We conclude that~$T$ is~$\equiv$-admissible.

Assume now that~$T$ is a $\equiv$-admissible separating tree, let~$\sigma$ be any linear extension of~$T$, and let~$T'$ be the tree obtained by insertion of $\sigma$ according to \cref{subsec:insertionMap}.
By the previous paragraph, we obtain that~$T'$ is a $\equiv$-admissible separating tree.
As~$\sigma$ is a linear extension of both~$T$ and~$T'$, we obtain by \cref{lem:admissibleSeparatingTreesWithCommonLinearExtensions} that~$T = T'$, and we conclude that~$T$ is the Hasse diagram of a $\equiv$-poset.
\end{proof}

\begin{corollary}
The $\equiv$-admissible separating trees are in bijection with the maximal cones of the quotient fan~$\quotientFan$ and with the vertices of the quotientope~$\Quotientope$.
\end{corollary}

\enlargethispage{.3cm}
\cref{prop:admissibleSeparatingTrees} characterizes the separating trees~$T$ which are (Hasse diagrams of) $\equiv$-posets for a given essential lattice congruence~$\equiv$.
Conversely, we can also describe the essential lattice congruences~$\equiv$ for which a given separating tree~$T$ is (the Hasse diagram of) a $\equiv$-poset.

\begin{proposition}
\label{prop:whichCongruences}
The simple essential lattice congruences~$\equiv$ for which a given separating tree~$T$ is (the Hasse diagram of) a $\equiv$-poset form an interval~$[\equiv_\circ, \equiv_\bullet]$ in the simple congruence lattice, where
\begin{itemize}
\item $\equiv_\circ$ is the congruence where~$\arcs_n \ssm \arcs_{\equiv_\circ}$ is the upper ideal of the arc poset generated by all up and down arcs which are not subarcs of any mandatory arc of~$T$,
\item $\equiv_\bullet$ is the congruence where~$\arcs_n \ssm \arcs_{\equiv_\bullet}$ is the upper ideal of the arc poset generated by the forbidden arcs of~$T$.
\end{itemize}
\end{proposition}

\begin{proof}
Immediately follows from \cref{def:admissibleSeparatingTrees,prop:admissibleSeparatingTrees}.
\end{proof}

\begin{example}
For instance, \cref{fig:minMaxCongruencesSeparatingTrees} shows the subarc minimal uncontracted arcs of the minimal (blue, top) and maximal (red, bottom) simple congruences for the four separating trees of \cref{fig:separatingTrees}.

\begin{figure}
	\capstart
	\centerline{\includegraphics[scale=.85]{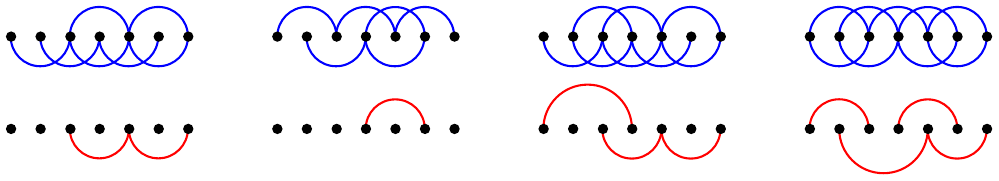}}
	\caption{The subarc minimal uncontracted arcs of the minimal (blue, top) and maximal (red, bottom) simple congruences for the four separating trees of \cref{fig:separatingTrees} (whose mandatory and forbidden arcs are shown in \cref{fig:mandatoryForbiddenArcsSeparatingTrees}).}
	\label{fig:minMaxCongruencesSeparatingTrees}
\end{figure}
\end{example}

\begin{remark}
\label{rem:separatingTreesArePermutrees}
As a follow-up of \cref{rem:separatingTreesVSPermutrees}, we obtain that any separating tree appears in at least one permutree congruence.
Indeed, the minimal simple congruence~$\equiv_\circ$ in \cref{prop:whichCongruences} is always a permutree congruence.
Indeed, consider an up arc~$u(i, j)$ with~$1 \le i < j \le n$ with~$j-i > 2$, and assume that both~$u(i,j-1)$ and~$u(i+1,j)$ are in~$\arcs_{\equiv_\circ}$.
By definition, $u(i,j-1)$ and~$u(i+1,j)$ are subarcs of some mandatory arcs~$\alpha$ and~$\beta$ of~$T$.
Since~$i-j > 2$ and~$\alpha$ and~$\beta$ are not crossing (see \cref{rem:separatingTreeFallHorizontalAxis}), either~$\alpha$ passes above~$j$ or~$\beta$ passes above~$i$.
Hence, $u(i,j)$ is a subarc of either~$\alpha$ or~$\beta$, thus it is in~$\arcs_{\equiv_\circ}$.
We conclude that the subarc minimal arcs in~$\arcs \ssm \arcs_{\equiv_\circ}$ have length at least~$2$, so that~$\equiv_\circ$ is a permutree congruence (see \cref{rem:permutreeCongruences}).
\end{remark}


\section{Schröder separating trees and preposets of simple congruences}
\label{sec:SchroderSeparatingTreesPreposetsSimpleCongruences}

In this section, we introduce Schröder separating trees (\cref{subsec:SchroderSeparatingTrees}) and prove that they provide a combinatorial description of all faces of the simple quotientopes (\cref{subsec:preposetsSimpleCongruences}).


\subsection{Schröder separating trees}
\label{subsec:SchroderSeparatingTrees}

We now define the higher dimensional analogue of separating trees.

\begin{definition}
\label{def:ShcroderSeparatingTree}
A \defn{Schröder separating tree} is a directed tree on the parts of a set partition of~$[n]$ such that any two ancestor (resp.~descendant) subtrees of a node~$J$ are separated by an integer~$j \in J$, meaning that all integers that appear in a node of one subtree are smaller than~$j$, while all integers that appear in a node of the other subtree are larger than~$j$.
\end{definition}

\begin{example}
\label{exm:SchroderSeparatingTrees}
\cref{fig:SchroderSeparatingTrees} represents some Schröder separating trees. We use the same conventions as in \cref{exm:separatingTrees}.
\begin{figure}[h]
	\capstart
	\centerline{\includegraphics[scale=.85]{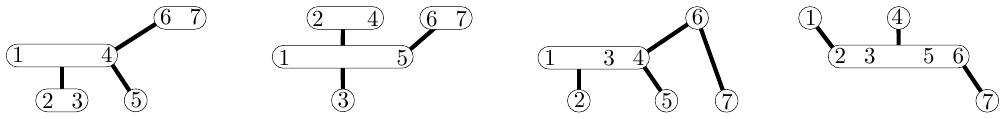}}
	\caption{Some Schröder separating trees on~$[7]$.}
	\label{fig:SchroderSeparatingTrees}
\end{figure}
\end{example}

\begin{remark}
Note that the separating trees are precisely the Schröder separating trees with $n$ nodes (or equivalently, where each node of~$S$ is a singleton).
\end{remark}

\begin{remark}
The number of Schröder separating trees on~$[n]$ with $k$ nodes is given by the following table:
\[
\begin{array}{c|cccccccccc}
k \backslash n & 1 & 2 & 3 & 4 & 5 & 6 & 7 & 8 & \dots \\
\hline
1 & 1 & 1 & 1 & 1 & 1 & 1 & 1 & 1 & \dots\\
2 & & 2 & 6 & 14 & 30 & 62 & 126 & 254 & \dots \\
3 & & & 8 & 46 & 184 & 638 & 2064 & 6438 & \dots\\
4 & & & & 42 & 388 & 2344 & 11818 & 54210 & \dots\\
5 & & & & & 264 & 3556 & 30134 & 207494 & \dots\\
6 & & & & & & 1898 & 35134 & 398040 & \dots\\ 
7 & & & & & & & 15266 & 372862 & \dots\\ 
8 & & & & & & & & 135668 & \dots \\
\vdots & & & & & & & & & \ddots \\
\hline
\Sigma & 1 & 3 & 15 & 103 & 867 & 8499 & 94543 & 1174967 & \dots
\end{array}
\]
\end{remark}

\begin{definition}
\label{def:contraction}
The \defn{contraction} of an edge~$e$ in a Schröder separating tree~$S$ is the tree obtained by contracting the edge~$e$ and labeling the resulting node with the union of the parts labelling the endpoints~of~$e$.
\end{definition}

\begin{lemma}
\label{lem:contractionSchroderSeparatingTree}
Schröder separating trees are stable by contraction.
\end{lemma}

\begin{proof}
Consider the contraction~$T$ of an edge~$J_\circ \to J_\bullet$ in a Schröder separating tree~$S$, and denote by~$J_{\circ\bullet} \eqdef J_\circ \sqcup J_\bullet$ the merged node of~$T$.
By definition, $T$ is a directed tree on the parts of a partition of~$[n]$.
Consider now two ancestor subtrees~$X$ and~$Y$ of a node~$J$ of~$T$ (the case of descendant subtrees is symmetric).
If~$J \ne J_{\circ\bullet}$, we have nothing to prove, so we assume that~$J = J_{\circ\bullet}$.
If~$X$ and~$Y$ were both ancestor subtrees of~$J_\circ$ (or similarly of~$J_\bullet$), there is nothing to prove.
Hence, assume that~$X$ was an ancestor subtree of~$J_\circ$ while~$Y$ was an ancestor subtree of~$J_\bullet$.
Then~$Y$ is contained in the ancestor subtree~$Z$ of~$J_\circ$ in~$S$ containing~$J_\bullet$.
As~$S$ is a Schröder separating tree, there is~$j \in J_\circ$ separating~$X$ from~$Z$.
As~$Y \subset Z$, the subtrees~$X$ and~$Y$ of~$T$ are separated by~$j \in J_{\circ\bullet}$.
We conclude that~$T$ is indeed a Schröder separating tree.
\end{proof}

\begin{lemma}
\label{lem:decontractionSchroderSeparatingTree}
The contraction minimal Schröder separating trees are precisely the separating trees.
\end{lemma}

\begin{proof}
A separating tree is clearly a contraction minimal Schröder separating tree since all nodes are singletons.
Conversely, consider a Schröder separating tree~$T$ which is not a separating tree.
Then it has a node~$J_{\circ\bullet}$ with~$|J_{\circ\bullet}| > 1$.
Let~$J_{\circ\bullet} = J_\circ \sqcup J_\bullet$ for some arbitrary~$J_\circ \ne \varnothing \ne J_\bullet$ with~$J_\circ < J_\bullet$.
Let~$S$ be the tree obtained by splitting the node~$J_{\circ\bullet}$ of~$T$ into two nodes~$J_\circ$ and~$J_\bullet$, adding an edge~$J_\circ \to J_\bullet$, and attaching each ancestor (resp.~descendant) subtree~$X$ of~$J_{\circ\bullet}$ to~$J_\circ$ if~$X < \max(J_\circ)$ (resp.~if~$\min(J_\bullet) \not< X$) and to~$J_\bullet$ otherwise.
By definition, $S$ is a tree on the parts of a partition of~$[n]$, and $T$ is the contraction of the edge~$J_\circ \to J_\bullet$ in the tree~$S$, so we just need to check that~$S$ is separating.
Consider two ancestor subtrees~$X$ and~$Y$ of a node~$J$ of~$S$ (the case of descendant subtrees is symmetric).
If~$J \notin \{J_\circ, J_\bullet\}$, there is nothing to prove.
Assume first that~$J = J_\bullet$.
We thus obtain that $X$ and~$Y$ were ancestor subtrees of~$J_{\circ\bullet}$ in~$T$.
As~$T$ is a Schröder separating tree, there is~$j \in J_{\circ\bullet}$ separating~$X$ and~$Y$.
As~$X$ and~$Y$ were attached to~$J_\bullet$, we have~$X \not < \max(J_\circ)$ and~$Y \not< \max(J_\circ)$, from which we deduce that~$j \in J_\bullet = J$.
Assume now that~$J = J_\circ$.
If neither~$X$ nor~$Y$ contains~$J_\bullet$, then they were both ancestor subtrees of~$J_{\circ\bullet}$ in~$T$.
As~$T$ is a Schröder separating tree, there is~$j \in J_{\circ\bullet}$ separating~$X$ and~$Y$.
As~$X$ and~$Y$ were attached to~$J_\circ$, we have~$X < \max(J_\circ)$ and~$Y < \max(J_\circ)$, from which we deduce that~$j \in J_\circ = J$.
Assume now that~$X$ does not contain~$J_\bullet$ but that~$Y$ contains~$J_\bullet$.
We have~$X < \max(J_\circ)$ so that we can define~$j_\circ \eqdef \min\set{j \in J_\circ}{X < j}$.
Any ancestor subtree~$Z$ of~$J_\bullet$ in~$S$ is an ancestor subtree of~$J_{\circ\bullet}$ in~$T$, hence there is~$j_Z \in J_{\circ\bullet}$ with~$X < j_Z < Z$, which implies that~$X < j_\circ < Z$.
Any descendant subtree~$Z$ of~$J_\bullet$ in~$Y$ satisfies~$\min(J_\bullet) < Z$ by definition of~$S$, so that~$X < j_\circ < Z$.
Finally, as $X < j_\circ < J_\bullet$, we obtain that~$j_\circ \in J_\circ$ separates~$X$ and~$Y$.
We thus conclude that~$S$ is indeed a Schröder separating tree.
\end{proof}

The following statement is the analogue of \cref{lem:separatingTree1}.

\begin{lemma}
\label{lem:SchroderSeparatingTree1}
Consider three vertices~$I,J,K$ in a Schröder separating tree~$S$ such that there is~$i \in I$, $j \in J$ and~$k \in K$ with~$i < j < k$.
If~$I = K$ or~$I - K$ is an edge of~$S$, then there is a directed path in~$S$ joining either~$\{I,K\}$ to~$J$, or~$J$ to~$\{I,K\}$.
\end{lemma}

\begin{proof}
The proof is identical to that of~\cref{lem:separatingTree1}.
As~$S$ is a tree, there is a (\apriori undirected) path~$\pi$ between~$\{I,K\}$ and~$J$ in~$S$.
If~$\pi$ is not directed, then there exists some node~$L$ along~$\pi$ such that~$\pi$ has either two incoming or two outgoing edges at~$L$.
This would contradict the separating property at~$L$ since~$i < j < k$.
\end{proof}

The following  definitions are the analogues of \cref{def:mandatoryArcsSeparatingTree,def:forbiddenArcsSeparatingTree}.

\begin{definition}
\label{def:mandatoryArcsSchroderSeparatingTree}
The \defn{mandatory arcs} of a Schröder separating tree~$S$ are the arcs of the \linebreak form~$(i, k, A, B)$, where~$i < k$ and~$A \sqcup B = {]i,k[}$ are such that
\begin{itemize}
\item the nodes~$I$ and~$K$ of~$S$ with~$i \in I$ and~$k \in K$ either coincide or form an edge in~$S$, and moreover~${]i,k[} \cap (I \cup K) = \varnothing$, 
\item $A$ (resp.~$B$) is the set of integers~$j \in {]i,k[}$ contained in a node~$J$ of~$S$ such that there is a directed path in~$S$ joining~$J$ to~$\{I,K\}$ (resp.~$\{I,K\}$ to~$J$).
\end{itemize}
\end{definition}

\begin{definition}
\label{def:forbiddenArcsSchroderSeparatingTree}
The \defn{forbidden up} (resp.~\defn{down}) \defn{arcs} of a Schröder separating tree~$S$ are the up arcs~$u(M, m)$ (resp.~down arcs~$d(M, m)$), with~$M \eqdef \max(L)$ and~$m \eqdef \min(R)$ where~$L$ and~$R$ are two ancestor (resp.~descendant) subtrees of a node~$J$ of~$S$ with~$L < R$.
\end{definition}

\begin{remark}
Note that the number of mandatory arcs of a Schröder separating tree on~$[n]$ is at least~$n-1$ but can be larger.
\end{remark}

\begin{example}
For instance, \cref{fig:mandatoryForbiddenArcsSchroderSeparatingTrees} shows the mandatory (blue, top) and forbidden (red, bottom) arcs of the four Schröder separating trees of \cref{fig:SchroderSeparatingTrees}.
\begin{figure}
	\capstart
	\centerline{\includegraphics[scale=.85]{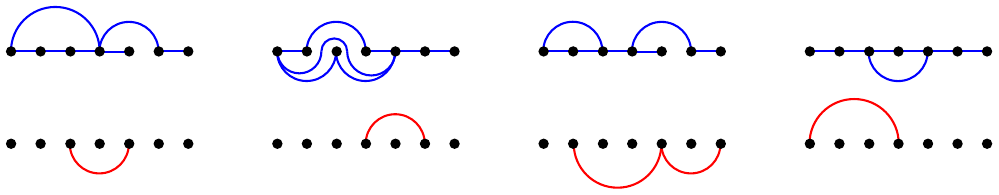}}
	\caption{The mandatory arcs (blue, top) and the forbidden arcs (red, bottom) of the four Schröder separating trees of \cref{fig:SchroderSeparatingTrees}.}
	\label{fig:mandatoryForbiddenArcsSchroderSeparatingTrees}
\end{figure}
\end{example}


\subsection{Preposets of simple congruences}
\label{subsec:preposetsSimpleCongruences}

The following statement is the analogue of \cref{prop:simpleImpliesSeparatingTrees}.

\begin{proposition}
\label{prop:simpleImpliesSchroderSeparatingTrees}
If~$\equiv$ is a simple essential congruence of the weak order, then the Hasse diagram of any $\equiv$-preposet is a Schröder separating tree.
\end{proposition}

\begin{proof}
As~$\equiv$ is a simple congruence, the Hasse diagrams of the $\equiv$-preposets are precisely the (iterated) contractions of the Hasse diagrams of the $\equiv$-posets.
The statement thus immediately follows from \cref{prop:simpleImpliesSeparatingTrees,lem:contractionSchroderSeparatingTree,lem:decontractionSchroderSeparatingTree}.
\end{proof}

We now aim at characterizing the $\equiv$-preposets of a simple congruence~$\equiv$, analoguously to \cref{def:admissibleSeparatingTrees}.
Recall that we have defined the mandatory and forbidden arcs of a Schröder separating tree in \cref{def:mandatoryArcsSeparatingTree,def:forbiddenArcsSeparatingTree}.

\begin{definition}
\label{def:admissibleSchroderSeparatingTrees}
Fix a simple essential lattice congruence~$\equiv$ of the weak order.
A Schröder separating tree~$S$ is \defn{$\equiv$-admissible} if~$\arcs_\equiv$ contains all mandatory arcs of~$S$ and none of the forbidden arcs of~$S$.
\end{definition}

\begin{example}
For instance, the Schröder separating trees which are admissible for the sylvester congruence (resp.~$\decoration$-permutree congruence) are precisely the Schröder trees (resp.~Schröder  \mbox{$\decoration$-per}\-mutrees) on~$[n]$ to which we have deleted all pending leaves to keep only the internal nodes and internal edges.
\end{example}

\begin{remark}
Note that the $\equiv$-admissible separating trees are precisely the $\equiv$-admissible Schröder separating trees with $n$ nodes (or equivalently, where each node of~$S$ is a singleton).
\end{remark}

The following statements are analogues of \cref{lem:contractionSchroderSeparatingTree,lem:decontractionSchroderSeparatingTree} for $\equiv$-admissible Schröder separating trees.

\begin{lemma}
\label{lem:contractionAdmissibleSchroderSeparatingTree}
$\equiv$-admissible Schröder separating trees are stable by contraction.
\end{lemma}

\begin{proof}
Consider the contraction~$T$ of an edge~$J_\circ \to J_\bullet$ in a $\equiv$-admissible Schröder separating tree~$S$, and denote by~$J_{\circ\bullet} \eqdef J_\circ \sqcup J_\bullet$ the merged node of~$T$.
We have seen in \cref{lem:contractionSchroderSeparatingTree} that~$T$ is a Schröder separating tree, so we just need to check that~$T$ is $\equiv$-admissible.

Consider first a mandatory arc~$(i,k,A,B)$ of~$T$.
Let~$I$ and~$K$ be the nodes of~$T$ such that~$i \in I$ and~$k \in K$, and similarly let~$I'$ and~$K'$ be the nodes of~$S$ such that~$i \in I'$ and~$k \in K'$.
Note that~$I' \subseteq I$ and~$K' \subseteq K$.
By \cref{def:mandatoryArcsSchroderSeparatingTree}, $I$ and~$K$ either coincide or form an edge of~$T$, and~${{]i,k[} \cap (I \cup K) = \varnothing}$.
Assume first that~$I'$ and~$K'$ coincide or form an edge of~$S$.
Since~${{]i,k[} \cap (I' \cup K') \subseteq {]i,k[} \cap (I \cup K) = \varnothing}$ and any directed path in~$S$ contracts to a directed path in~$T$, we then obtain that~$(i,k,A,B)$ is also a mandatory arc of~$S$, hence belongs to~$\arcs_\equiv$ since~$S$ is $\equiv$-admissible.
Assume now that~$I'$ and~$K'$ are distinct and there is no edge between~$I'$ and~$K'$ in~$S$.
Up to symmetry, we can then assume that~$I = J_{\circ\bullet}$, that~$I' = J_\circ$ and that~$K = K'$ is a neighbor of~$J_\bullet$ in~$S$.
Since~${]i,k[} \cap (I \cup K) = \varnothing$, we have~$i < k < j$ or~$j < i < k$  for all~$j \in J_\bullet$.
Since~$S$ is a Schröder separating tree and contains the edge~$I' = J_\circ \to J_\bullet$, it cannot contain the edge~$K' \to J_\bullet$, as otherwise~$I'$ and~$K'$ would be contained in two descendant subtrees of~$J_\bullet$ not separated by~$J_\bullet$.
Hence we obtain that~$S$ contains the path~$I' \to J_\bullet \to K'$.
Assume for instance that there is~$j \in J_\bullet$ with~$i < k < j$.
As~$i \in I'$ and~$j \in J_\bullet$ and~$I' \to J_\bullet$ in~$S$, we obtain an arc~$(i, j, A', B')$ of~$S$ for some~$A' \sqcup B' = {]i,j[}$, which is thus in~$\arcs_\equiv$ since $S$ is $\equiv$-admissible.
As~$i < k < j$ and any directed path in~$S$ contracts to a directed path in~$T$, we obtain that~$(i,k,A,B)$ is a subarc of~$(i,j,A',B')$, thus also belongs to~$\arcs_\equiv$.

Consider now a forbidden up arc~$u(M,m)$ of~$T$.
Let~$J$ be the node of~$T$ with two ancestor subtrees~$L$ and~$R$ with~$L < R$ such that~$M = \max(L)$ and~$m = \min(R)$.
If~$L$ and~$R$ are both ancestor subtrees of the same node of~$S$, then the up arc~$u(M,m)$ is not in~$\arcs_\equiv$ since~$S$ is \mbox{$\equiv$-admissible}.
Otherwise, we have that~$J = J_{\circ\bullet}$, and $L$ is an ancestor subtree of~$J_\circ$ while~$R$ is an ancestor subtree of~$J_\bullet$ (or the opposite, which is similar).
Let~$R'$ be the ancestor subtree of~$J_\circ$ in~$S$ containing~$R$, and let~$m' \eqdef \max(R')$.
As~${R \subseteq R'}$, we have $m' < m$, so that the up arc~$u(M, m')$ is a subarc of the arc~$u(M,m)$.
We conclude that the arc~$u(M,m)$ is not in~$\arcs_\equiv$ since its subarc~$u(M, m')$ is not in~$\arcs_\equiv$ since~$S$ is $\equiv$-admissible.
The proof is symmetric for the forbidden down arcs.
We conclude that~$T$ is indeed $\equiv$-admissible.
\end{proof}

\begin{lemma}
\label{lem:decontractionAdmissibleSchroderSeparatingTree}
The contraction minimal $\equiv$-admissible Schröder separating trees are precisely the $\equiv$-admissible separating trees.
\end{lemma}

\begin{proof}
A $\equiv$-admissible separating tree is clearly a contraction minimal $\equiv$-admissible Schröder separating tree since all nodes are singletons.
Conversely, consider a $\equiv$-admissible Schröder separating tree~$T$ which is not a $\equiv$-admissible separating tree.
Then it has a node~$J_{\circ\bullet}$ with~$|J_{\circ\bullet}| > 1$.
Let~$J_{\circ\bullet} = J_\circ \sqcup J_\bullet$ for some arbitrary~${J_\circ \ne \varnothing \ne J_\bullet}$ with~$J_\circ < J_\bullet$.
Let~$S$ be the tree obtained by splitting the node~$J_{\circ\bullet}$ of~$S$ into two nodes~$J_\circ$ and~$J_\bullet$, adding an edge~$J_\circ \to J_\bullet$, and attaching each ancestor subtree~$X$ of~$J_{\circ\bullet}$ to~$J_\circ$ if~$X < \max(J_\circ)$ and the up arc~$u(\max(X), \min(J_\bullet))$ is not in~$\arcs_\equiv$ and to~$J_\bullet$ otherwise, and similarly each descendant subtree~$X$ of~$J_{\circ\bullet}$ to~$J_\bullet$ if~$\min(J_\bullet) < X$ and the down arc~$d(\max(J_\circ), \min(X))$ is not in~$\arcs_\equiv$ and to~$J_\circ$ otherwise.
Note that our definition of~$S$ here slightly differs from that of the proof of \cref{lem:decontractionSchroderSeparatingTree} since we need to additionally ensure that~$S$ is still $\equiv$-admissible.
By definition, $S$ is a tree on the parts of a partition of~$[n]$, and $T$ is the contraction of the edge~$J_\circ \to J_\bullet$ in the tree~$S$, so we just need to check that~$S$ is separating and $\equiv$-admissible.

We first prove that~$S$ is separating. 
Consider two ancestor subtrees~$X$ and~$Y$ of a node~$J$ of~$S$ (the case of descendant subtrees is symmetric).
We want to prove that there is~$j \in J$ separating~$X$ from~$Y$.
The proof is identical to that of \cref{lem:decontractionSchroderSeparatingTree} if~$J \ne J_\bullet$.
Assume now that~$J = J_\bullet$.
We thus obtain that $X$ and~$Y$ were ancestor subtrees of~$J_{\circ\bullet}$ in~$T$.
As~$T$ is a Schröder separating tree, there is~$j \in J_{\circ\bullet}$ separating~$X$ and~$Y$.
Assume by symmetry that~$X < j < Y$.
As~$T$ is $\equiv$-admissible, the forbidden up arc~$u(\max(X), \min(Y))$ of~$T$ is not in~$\arcs_\equiv$.
Since~$X$ is attached to~$J_\bullet$, we have two options:
\begin{itemize}
\item either~$X \not< \max(J_\circ)$, then we obtain that~$j \in J_\bullet = J$ and~$j$ separates~$X$ from~$Y$,
\item or~$X < \max(J_\circ)$ and the up arc~$u(\max(X), \min(J_\bullet))$ is in~$\arcs_\equiv$. As~$u(\max(X), \min(Y))$ is not in~$\arcs_\equiv$, we obtain that~$\min(J_\bullet) < \min(Y)$, so that~$\min(J_\bullet)$ separates~$X$ from~$Y$.
\end{itemize}
We conclude that~$S$ is indeed a Schröder separating tree.

Consider now a mandatory arc~$(i,k,A,B)$ of~$S$.
Let~$I$ and~$K$ be the nodes of~$S$ such that~$i \in I$ and~$k \in K$, and similarly let~$I'$ and~$K'$ be the nodes of~$T$ such that~$i \in I'$ and~$k \in K'$.
Note that~$I \subseteq I'$ and~$K \subseteq K'$.
Moreover, either~$I = I'$ or~$I \in \{J_\circ, J_\bullet\}$ and~$I' = J_{\circ\bullet}$ and similarly for~$K$ and~$K'$.
By \cref{def:mandatoryArcsSchroderSeparatingTree}, $I$ and~$K$ either coincide or form an edge of~$S$, and~${{]i,k[} \cap (I \cup K) = \varnothing}$.
Hence, $I'$ and~$K'$ either coincide or form an edge of~$T$.
If we have~${]i,k[} \cap (I' \cup K') = \varnothing$, then we obtain that $(i,k,A,B)$ is a mandatory arc of~$T$ since any directed path in~$S$ contracts to a directed path in~$T$.
As $T$ is $\equiv$-admissible, we therefore obtain that~$(i,k,A,B)$ is in~$\arcs_\equiv$ as desired.
Hence, we can assume that~${]i,k[} \cap (I' \cup K') \ne \varnothing$.
Note that this cannot happen if~$I = I'$ and~$K = K'$ (since we would have~$I \cup K = I' \cup K'$), nor if~$I \ne I'$ and~$K \ne K'$ (because, if~$I \ne K$ then~$I \cup K = J_{\circ\bullet} = I' \cup K'$, while if~$I = K$ then~${]i,k[} \subseteq I \subseteq I'$ so that~${]i,k[} \cap (I' \cup K') = {]i,k[} \cap I = \varnothing$).
Hence, we can assume by symmetry that~$I = J_\circ$ and~$I' = J_{\circ\bullet}$ while~$K = K'$ is a neighbor of~$J_\circ$ in~$S$ distinct from~$J_\circ$ and~$J_\bullet$.
Moreover, we must have~$i < \min(J_\bullet) < k$, so that~$i = \max(J_\circ)$ and~$k = \min(K \ssm [i])$ since~${{]i,k[} \cap (I \cup K) = \varnothing}$.
Since~$\max(J_\circ) = i < k \in K$, the definition of~$S$ prevents~$K$ to be a parent of~$J_\circ$.
Hence, $K$ is a child of~$J_\circ$, and we let~$X$ be the descendant subtree of~$J_\circ$ containing~$K$, and~$m \eqdef \min(X)$.
Note that the down arc~$d(i,m)$ is a subarc of~$(i,k,A,B)$.
We moreover claim that it is always in~$\arcs_\equiv$.
Indeed, the definition of~$S$ leaves us with two possibilities:
\begin{itemize}
\item If~$\min(J_\bullet) \not< X$, then~$m < \min(J_\bullet)$, hence~$d(i,m)$ is a subarc of the mandatory arc of~$T$ defined by the two consecutive elements~$\max(J_\circ)$ and~$\min(J_\bullet)$ of~$J_{\circ\bullet}$, which is in~$\arcs_\equiv$ since~$T$ is $\equiv$-admissible.
\item If~$\min(J_\bullet) < X$ then~$d(i,m) = d(\max(J_\circ), \min(X))$ is in~$\arcs_\equiv$ by construction of~$S$.
\end{itemize}
We now consider two increasing sequences~$i = a_0 < \dots < a_r = k$ and~$i = b_0 < \dots < b_s = k$ of elements of~$[i,k]$, where
\begin{itemize}
\item $m = a_1 < \dots < a_r = k$ are such that for all~$j \in [r-1]$, the nodes~$A_j$ and~$A_{j+1}$ of~$T$ containing~$a_j$ and~$a_{j+1}$ either coincide or form an edge of~$T$, and~${{]a_j, a_{j+1}[} \cap (A_j \cup A_{j+1}) = \varnothing}$,
\item $b_1 < \dots < b_{s-1}$ denote the elements of~${]i,k[} \cap J_\bullet$.
\end{itemize}
Note that~$\{a_1, \dots, a_{r-1}\} \subseteq A$ while~$\{b_1, \dots, b_{s-1}\} \subseteq B$.
Moreover, $(a_{j-1}, a_j)$ defines a mandatory arc~$\alpha_j$ of~$T$ for each~$j \in [2,r]$ (because $a_{j-1}$ and~$a_j$ are consecutive either in a common node of~$T$ or in neighbors in~$T$).
Similarly, $(b_{j-1}, b_j)$ defines a mandatory arc~$\beta_j$ of~$T$ for each~$j \in [s]$ (because $b_{j-1}$ and~$b_j$ are consecutive in the node~$J_{\circ\bullet}$ of~$T$ for~$j \in [s-1]$, and $b_{s-1} \in J_{\circ\bullet}$ while $b_s = k \in K$ are consecutive in neighbors of~$T$).
Moreover, all these arcs are subarcs of~$(i,k,A,B)$ and belong to~$\arcs_\equiv$ since~$T$ is $\equiv$-admissible.
Finally, we have seen that~$\alpha_1 \eqdef d(a_0,a_1) = d(i,m)$ is a subarc of~$(i,k,A,B)$ and belongs to~$\arcs_\equiv$.
Assume now by means of contradiction that~$(i,k,A,B)$ is not in~$\arcs_\equiv$.
Let~$i \le p < q \le k$ be such that the subarc~$\gamma \eqdef (p, q, A \cap {]p,q[}, B \cap {]p,q[})$ of~$(i,k,A,B)$ is minimal in~$\arcs \ssm \arcs_\equiv$ for the subarc order.
There is no~$j \in [r]$ (resp.~$j \in [s]$) such that~${a_{j-1} \le p < q \le a_j}$ (resp.~$b_{j-1} \le p < q \le b_j$) since~$\alpha_j$ (resp.~$\beta_j$) is a subarc of~$(i,k,A,B)$ and belongs to~$\arcs_\equiv$.
Hence, there is~$j \in [r-1]$ (resp.~$j \in [s-1]$) such that~$p < a_j < q$ (resp.~$p < b_j < q$).
Since~$\{a_1, \dots, a_{r-1}\} \subseteq A$ while~$\{b_1, \dots, b_{s-1}\} \subseteq B$, the subarc~$\gamma$ of~$(i,k,A,B)$ is neither an up arc, nor a down arc.
This contradicts our assumption that~$\equiv$ is a simple congruence.
We finally conclude that the mandatory arc~$(i,k,A,B)$ indeed belongs to~$\arcs_\equiv$.

Consider now a forbidden up arc~$u(M,m)$ of~$S$.
Let~$J$ be the node of~$S$ with two ancestor subtrees~$L$ and~$R$ with~$L < R$ such that~$M = \max(L)$ and~$m = \min(R)$.
If~$L$ and~$R$ are both ancestor subtrees of the same node of~$T$, then the up arc~$u(M,m)$ is not in~$\arcs_\equiv$ since~$T$ is \mbox{$\equiv$-admissible}.
Hence, we can assume that~$J = J_\circ$ and~$L \not\supseteq J_\bullet$ while~$R \supseteq J_\bullet$.
Since~$L$ is attached to~$J_\circ$, the up arc~$u(\max(L), \min(J_\bullet))$ is not in~$\arcs_\equiv$.
We can thus assume that~$m < \min(J_\bullet)$.
Let~$Y$ be the ancestor subtree of~$J_\bullet$ containing~$m$.
Then~$L$ and~$Y$ are distinct ancestor subtrees of~$J_{\circ\bullet}$ in~$T$ with~$L < Y$ and~$M = \max(L)$ while~$m = \min(Y)$, so that~$u(M,m)$ is a forbidden up arc of~$T$.
Since~$T$ is $\equiv$-admissible, we conclude that~$u(M,m)$ is not in~$\arcs_\equiv$.
The proof is symmetric for the forbidden down arcs.
We conclude that~$S$ is indeed $\equiv$-admissible.
\end{proof}

\begin{proposition}
\label{prop:admissibleSchroderSeparatingTrees}
If~$\equiv$ is a simple essential congruence of the weak order, then the Hasse diagrams of the $\equiv$-preposets are precisely the $\equiv$-admissible Schröder separating trees.
\end{proposition}

\begin{proof}
As $\equiv$ is a simple congruence, the Hasse diagrams of the $\equiv$-preposets are precisely the (iterated) contractions of the Hasse diagrams of the $\equiv$-posets.
The statement thus follows from \cref{prop:admissibleSeparatingTrees,lem:contractionAdmissibleSchroderSeparatingTree,lem:decontractionAdmissibleSchroderSeparatingTree}.
\end{proof}

\begin{corollary}
The contraction poset on $\equiv$-admissible Schröder separating trees is isomorphic to the face lattice of the quotientope~$\Quotientope$ (or to the reversed face lattice of the quotient fan~$\quotientFan$).
\end{corollary}

\begin{remark}
The number of $\equiv$-admissible Schröder separating trees is
\[
\sum_T 2^{\#\des(T)}
\]
where the sum ranges over all $\equiv$-admissible separating trees, and~$\des(T)$ denotes the descents of~$T$ (\ie~the edges~$i \leftarrow j$ with~$1 \le i < j \le n$).
\end{remark}

\begin{remark}
\cref{prop:whichCongruences} extends verbatim to Schröder separating trees.
\end{remark}

\begin{example}
For instance, \cref{fig:minMaxCongruencesSchroderSeparatingTrees} shows the subarc minimal uncontracted arcs of the minimal (blue, top) and maximal (red, bottom) simple congruences for the four Schröder separating trees of \cref{fig:SchroderSeparatingTrees}.

\begin{figure}
	\capstart
	\centerline{\includegraphics[scale=.85]{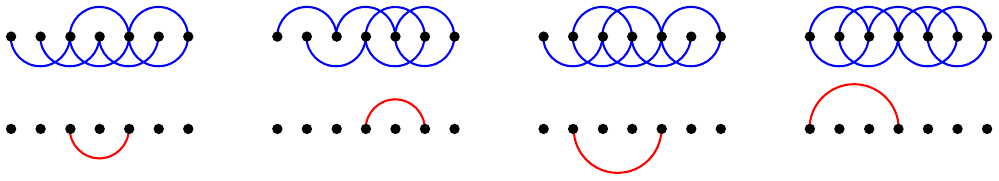}}
	\caption{The subarc minimal uncontracted arcs of the minimal (blue, top) and maximal (red, bottom) simple congruences for the four Schröder separating trees of \cref{fig:SchroderSeparatingTrees} (whose mandatory and forbidden arcs are shown in \cref{fig:mandatoryForbiddenArcsSchroderSeparatingTrees}).}
	\label{fig:minMaxCongruencesSchroderSeparatingTrees}
\end{figure}
\end{example}


\section{Algebraic structure}
\label{sec:algebraicStructure}

This section is devoted to algebraic aspects of the essential simple congruences.
Since there is no natural coproduct in this context, we shall only sketch the algebraic structure and refer the reader to~\cite{ChatelPilaud, PilaudPons-permutrees, Pilaud-arcDiagramAlgebra} for more detailed explanations and basic definitions.

The natural algebra structure between sets of permutations is given by the shuffle product of the Hopf algebra on permutations of C.~Reutenauer and C.~Malvenuto~\cite{MalvenutoReutenauer}.
Recall that the \defn{shifted shuffle} of two permutations~$\rho \in \f{S}_m$ and~$\sigma \in \f{S}_n$ is the set~$\rho \shiftedShuffle \sigma$ of permutations of~$\f{S}_{m+n}$ whose first $m$ (resp.~last~$n$) values appear in the same order as~$\rho$ (resp.~$\sigma$).
For instance,
\[
{\red 12} \shiftedShuffle {\darkblue 231} = \{ {\red 12}{\darkblue 453}, {\red 1}{\darkblue 4}{\red 2}{\darkblue 53}, {\red 1}{\darkblue 45}{\red 2}{\darkblue 3}, {\red 1}{\darkblue 453}{\red 2}, {\darkblue 4}{\red 12}{\darkblue 53}, {\darkblue 4}{\red 1}{\darkblue 5}{\red 2}{\darkblue 3}, {\darkblue 4}{\red 1}{\darkblue 53}{\red 2}, {\darkblue 45}{\red 12}{\darkblue 3}, {\darkblue 45}{\red 1}{\darkblue 3}{\red 2}, {\darkblue 453}{\red 12} \}.
\]

Given two essential simple congruences~$\equiv$ on~$\f{S}_m$ and~$\equiv'$ on~$\f{S}_n$ and two congruence classes~$R$ of~$\f{S}_m/{\equiv}$ and~$S$ of~$\f{S}_n/{\equiv'}$, we can consider the shuffle product~$R \shiftedShuffle S = \bigcup_{\rho \in R, \sigma \in S} \rho \shiftedShuffle \sigma$.
To provide an algebraic structure, we need~$R \shiftedShuffle S$ to be a disjoint union of classes of an essential simple congruence that depends on~$\equiv$ and~$\equiv'$.
There are various possible solutions, but we want our choice to specialize to the existing constructions for binary trees~\cite{LodayRonco}, Cambrian trees~\cite{ChatelPilaud}, and permutrees~\cite{PilaudPons-permutrees}.

To this end, let us define a \defn{birational sequence} of length $n$ as a sequence~$\b{r}$ of~$n$ pairs of non-negative rational numbers.
For instance, $((0,0), (\frac12,\frac12), (\frac13,\frac12), (\frac12,\frac12), (1,1))$ is a birational sequence of length~$5$.
We represent~$\b{r}$ by placing vertical walls above and below the points of~$[n]$, each labeled by a rational number.
We associate to~$\b{r}$ the simple congruence~$\equiv_\b{r}$ whose subarc minimal forbidden arcs are the shortest up or down arcs that cross walls whose label sums are greater than or equal to~$1$.
For example, the birational sequence~$((0,0), (\frac12,\frac12), (\frac13,\frac12), (\frac12,\frac12), (1,1))$ forbids the arcs $d(1, 4), u(1, 5), d(2, 5)$ and all their superarcs.
See \cref{fig:birationalSequence}.
Note that the first and last birationals of~$\b{r}$ do not matter at the moment since no arcs cross these walls.
We first observe that this model encompasses all essential simple congruences.

\begin{figure}[h]
	\capstart
	\centerline{\includegraphics[scale=1]{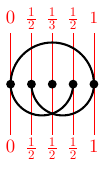}}
	\caption{Representation of the birational sequence $((0,0), (\frac12,\frac12), (\frac13,\frac12), (\frac12,\frac12), (1,1))$ and its corresponding minimal forbidden arcs $d(1, 4), u(1, 5), d(2, 5)$.}
	\label{fig:birationalSequence}
\end{figure}

\begin{proposition}
The essential simple congruences are precisely the congruences~$\equiv_\b{r}$ defined by birational sequences~$\b{r}$.
\end{proposition}

\begin{proof}
The forbidden arcs above and below the axis depending on disjoint sets of rationals, we can limit ourselves to proving that any set of nonnesting essential up arcs is precisely \linebreak the set of subarc minimal forbidden arcs by a rational sequence.
Consider a set of~$k$ \linebreak nonnesting essential up arcs~$u(a_1, b_1), \dots, u(a_k, b_k)$, ordered such that~${1 \le a_1 < \dots < a_k \le n}$ \linebreak and~${1 \le b_1 < \dots < b_k \le n}$ (these conditions are equivalent since the arcs are nonnesting).
We prove by induction on~$k$ that this set is indeed the set of subarc minimal forbidden arcs by some rational sequence of length~$n$, which moreover vanishes outside~${]a_1, b_k[}$.
For~$k = 0$, the sequence~$(0, \dots, 0)$ of length~$n$ forbids no up arc.
By induction hypothesis, there exists a rational sequence~$\b{r} \eqdef (r_1, \dots, r_n)$  which precisely forbids the arcs~$u(a_1,b_1), \dots, u(a_{k-1},b_{k-1})$ and vanishes outside~${]a_1, b_{k-1}[}$.
As the condition for preserving an arc is open, there exists a sufficiently small~${\varepsilon > 0}$ so that~$\b{r} + \varepsilon \one_{]a_k, b_k[}$ still forbids precisely~$u(a_1,b_1), \dots, u(a_{k-1},b_{k-1})$.
Finally, let~${\b{s} \eqdef \b{r} + \varepsilon \one_{]a_k, b_k[} + (1 - \varepsilon (b_k-a_k-1) - \sum_{a_k < i < b_{k-1}} r_i) \one_{\{b_k-1\}}}$.
By definition, $\b{s}$ vanishes outside~${]a_1,b_k[}$.
As~$\b{r} + \varepsilon \one_{]a_k, b_k[}$ precisely forbids~$u(a_1,b_1), \dots, u(a_{k-1},b_{k-1})$ and vanishes outside~${]a_1, b_k[}$, the sequence~$\b{s}$ also forbids these arcs and an additional subarc minimal arc of the form~$(a,b_k)$.
The choice of the weight at~$b_k-1$ ensures that~$a = a_k$.
\end{proof}

We shall now see that the concatenation of the birational sequences indeed defines an algebra structure for essential simple congruences.
We denote by~$\b{rs}$ the birational sequence of length~$m+n$ obtained by the concatenation of a birational sequence~$\b{r}$ of length~$m$ with a birational sequence~$\b{s}$ of length~$n$.

\begin{proposition}
\label{prop:product}
For any birational sequences~$\b{r}, \b{s}$ of length~$m$ and~$n$, and any congruence classes~$R$ of~$\f{S}_m/{\equiv_\b{r}}$ and~$S$ of~$\f{S}_n/{\equiv_\b{s}}$, the shifted shuffle~$R \shiftedShuffle S$ is a disjoint union of congruence classes of~$\f{S}_{m+n}/{\equiv_\b{rs}}$.
\end{proposition}

\begin{proof}
Observe first that the arcs forbidden by~$\b{rs}$ with both endpoints in~$[m]$ (resp.~in~$[m+n] \ssm [m]$) are precisely the arcs forbidden by~$\b{r}$ (resp.~the arcs forbidden by~$\b{s}$ shifted by~$m$).

Since the shifted shuffles of distinct permutations are disjoint, we only need to prove that, if a permutation~$\pi$ belongs to~$R \shiftedShuffle S$, then so do all the permutations of its class.
As classes are connected in the weak order, we only need to prove it for a permutation~$\pi'$ adjacent~to~$\pi$.

Let~$\rho \in R$ and~$\sigma \in S$ be such that~$\pi$ belongs to~$\rho \shiftedShuffle \sigma$.
If~$\pi$ and~$\pi'$ differ by the transposition of a value in~$[m]$ and a value in~$[m+n] \ssm [m]$, then $\pi'$ also belongs to~$\rho \shiftedShuffle \sigma$ by definition of the shifted shuffle~$\shiftedShuffle$.
If~$\pi$ and~$\pi'$ differ by the transposition of two values~$u,v$ in~$[m]$ (resp.~in~$[m+n] \ssm [m]$), then $\pi'$ belongs to~$\rho' \shiftedShuffle \sigma$ (resp.~$\rho \shiftedShuffle \sigma'$) where~$\rho$ and~$\rho'$ (resp.~$\sigma$ and~$\sigma'$) differ by the transposition of the values~$u,v$.
As the arc corresponding to the transposition between~$\pi$ in~$\pi'$ is forbidden by~$\b{rs}$, it (resp.~its shift) is also forbidden by~$\b{r}$ (resp.~$\b{s}$), so that~$\rho' \in R$ (resp.~$\sigma' \in S$).
In both cases, we obtained that~$\pi'$ also belongs to~$R \shiftedShuffle S$.
\end{proof}

It follows from \cref{prop:product} that one can define the following algebra.

\begin{definition}
A \defn{decorated separating tree} is a pair~$(T, \b{t})$, where~$\b{t}$ is a birational sequence and~$T$ is a separating tree of~$\equiv_\b{t}$.
The \defn{decorated separating tree algebra} is the algebra with basis indexed by decorated separating trees and where the product is given by~$(R, \b{r}) \cdot (S, \b{s}) = \sum_{i \in [k]} (T_i, \b{rs})$ such that the union of the linear extensions of~$T_1, \dots, T_k$ is the shifted shuffle of the linear extensions of~$R$ with the linear extensions of~$S$.
\end{definition}

\begin{remark}
\begin{enumerate}
\item Restricting to bibinary sequences (where all entries are either~$0$ or~$1$), we recover the algebraic structure defined on permutrees in~\cite{PilaudPons-permutrees}. It encompasses in particular the algebraic structure on binary trees by J.-L.~Loday and M.~Ronco~\cite{LodayRonco} and on Cambrian trees by~\cite{ChatelPilaud}.
\item Similarly to~\cite{ChatelPilaud, PilaudPons-permutrees}, we can understand the product of~\cref{prop:product} by over and under operations on decorated separating trees, and we obtain a multiplicative basis of the decorated separating tree algebra using upper sets of congruence classes. The only crucial ingredient is that all congruence classes are intervals of the weak order. Details are left to the reader.
\item In contrast to~\cite{LodayRonco, ChatelPilaud, PilaudPons-permutrees}, there is unfortunately no natural Hopf algebra structure on decorated separating trees. The natural coproduct would be the deconcatenation of permutations (with pairs of rationals attached to values). Unfortunately, this operation does not define a coalgebra on decorated separating trees. For example, the birational sequence~$((0,0), (\frac12,0), (\frac13,0), (\frac12,0), (0,0))$ has only one forbidden arc~$u(1,5)$. Hence, the permutation~$24153$ is alone in its class. So its deconcatenation gives rise to a single term~$2314 \otimes 1$ in~$\f{S}_4 \otimes \f{S}_1$. The corresponding deconcatenated birational sequence~$((0,0), (\frac12,0), (\frac12,0), (0,0))$ forbids~$u(1,4)$, so that~$2314$ is equivalent to~$2341$. Hence, the deconcatenation does not define a coalgebra.
\item Restricting to bibinary sequences, this phenomenon disappears, explaining the Hopf algebra structures in~\cite{LodayRonco, ChatelPilaud, PilaudPons-permutrees}.
\end{enumerate}
\end{remark}


\section{Quiver representation theory}
\label{sec:representationTheory}

In this section we describe how our main results apply to the combinatorics of quiver representations~\cite{AssemSimsonSkowronski, Schiffler, AdachiIyamaReiten, DemonetIyamaReadingReitenThomas} (see also the survey articles~\cite{Thomas-TamariQuiverRepresentations, Thomas-surveyTorsionClasses}).

Fix a field $K$.
A quiver $Q \eqdef (Q_0,Q_1)$ is a finite directed graph with vertices~$Q_0$ and directed edges (or arrows)~$Q_1$.
The \defn{path algebra} $KQ$ of $Q$ is the $K$-algebra whose underlying $K$-vector space has a basis indexed by all directed paths in $Q$, and where the product of two paths is their concatenation when it is a path, or zero otherwise.

Let $I$ be a two-sided ideal of the path algebra $KQ$. 
A \defn{representation} of $(Q,I)$ is an assignment of a finite dimensional $K$-vector space~$V_i$ for each vertex~$i\in Q_0$, and an assignment of a linear map~$\phi_\alpha: V_i \to V_j$ for each arrow~$\alpha$ from~$i$ to~$j$ in~$Q_1$ such that relations in $I$ are satisfied.
When the ideal~$I$ is generated by paths in~$Q$, this means that whenever a path $w$ (which can be thought of as a word in the arrows of~$Q_1$) belongs to~$I$, the composition of the corresponding linear maps is zero.

There is a $K$-linear equivalence of 
the category of representations of $(Q,I)$ over $K$ and the category of finitely generated modules over $A \eqdef KQ/I$.
This equivalence allows us to discuss morphisms, quotients and extensions of a representation (by thinking of the representation as a module over $A$).

As combinatorialists, quiver representations are of interest because certain subcategories of quiver representations, ordered by containment, form interesting and rich lattices.
In particular, we focus on subcategories called \defn{torsion classes}.
A subcategory $\mathcal{T}$ is a \defn{torsion class} if it is closed under quotients and extensions. In particular, the empty subcategory, and the subcategory of all representations are both torsion classes.
Clearly, if $\mathcal{U}$ and $\mathcal{T}$ are torsion classes, then so is their intersection.
Thus, the set of torsion classes forms a lattice which we denote by $\tors(A)$.

In fact, the lattice $\tors(A)$ shares many pleasant properties with the weak order on a finite Coxeter group.
The lattice $\tors(A)$ is semidistributive, congruence uniform, and Hasse-regular, although it is not generally known whether $\tors(A)$ can be realized as the skeleton of a simple polytope.
One notable exception is when A is the preprojective algebra $\Pi_n$ of type~$A_n$.
In this paper, we work with a simpler algebra, which we denote as $\RA_n$ following~\cite{BarnardCarrollZhu}, also called the \defn{Orpheus Algebra} in~\cite{BarnardCoelhoSimoesGunawanSchiffler}.
Importantly, the lattices of torsion classes for both the Orpheus algebra $\RA_n$ \cite[Sect.~4]{BarnardCarrollZhu} and the preprojective algebra of type $A_n$ \cite[Thm.~2.3]{Mizuno, Mizuno-arcDiagrams} are isomorphic to the weak order on $A_n$.

\begin{definition}
The quiver $Q \eqdef (Q_0,Q_1)$ of $\RA_n$ has vertices $Q_0 = \{1,2,\ldots n\}$, arrows $Q_1$ given by $\alpha_i: i \to i+1$ and $\beta_i:i+1\to i$ for all $i\in [n-1]$, and two-sided ideal $I$ generated by all cycles of the form $\alpha_i\beta_{i} = \beta_i\alpha_i$.
\end{definition}

In \cite{DemonetIyamaReadingReitenThomas}, the authors study the lattice of congruences of the lattice~$\tors A$ of torsion classes, where~$A$ is any finite-dimensional algebra.
Their main result gives a correspondence between quotients $A/I'$ of $A$ by an ideal $I'$ and certain \defn{algebraic} congruences of $\tors A$.
In the particular case where $A$ is the Orpheus algebra~$\RA_n$, these algebraic congruences actually characterize simple congruences.

Now that we know that every simple congruence $\equiv$ of the weak order on $A_n$ is an algebraic congruence of the algebra $\RA_n$ by some ideal $I_\equiv$, it is natural to ask for a set of minimal generators~$w$ of $I_{\equiv}$, in the sense that no other generator contains $w$ as a consecutive subword.
Indeed, these generators can be read off as minimal arcs (in the sense of subarcs).

\enlargethispage{.3cm}
For each pair~$1 \le i < j \le n$, we earlier defined the up arc~$u(i,j) \eqdef (i, j, {]i,j[}, \varnothing)$ and the down arc~$u(i,j) \eqdef (i, j, \varnothing, {]i,j[})$, and we now define the up path~$up(i,j) \eqdef \alpha_i \dots \alpha_{j-1}$ and  the down path~$dp(i,j) \eqdef \beta_{j-1} \dots \beta_i$ in the quiver of~$\RA_{n-1}$.

\begin{theorem}
Consider any simple congruence~$\equiv$ of the weak order on $\f{S}_n$, and let~$\arcs_\equiv$ be its arc ideal and $I_\equiv$ be its path algebra ideal of~$\RA_n$. 
Then $u(i,j)$ (resp.~$d(i,j)$) is a subarc minimal arc of~$\arcs \ssm \arcs_\equiv$ if and only if~$up(i,j)$ (resp.~$dp(i,j)$) is a minimal generator of $I_{\equiv}$.
\end{theorem}

\begin{corollary}
For any simple congruence~$\equiv$ of the weak order on~$\f{S}_n$, the $\b{g}$-vector fan of~$\RA_{n-1}/I_\equiv$ is (up to a change of basis) the intersection of the hyperplane~$\set{\b{x} \in \R^n}{\sum_{i=1}^n x_i = 0}$ with the fan of~$\R^n$ with cones~$C(S) \eqdef \set{\b{x} \in \R^n}{x_i \le x_j \text{ for } i \to j \text{ in } S}$ defined by the $\equiv$-admissible Schröder separating trees~$T$ (and maximal cones correspond to $\equiv$-admissible separating trees).
\end{corollary}


\addtocontents{toc}{\vspace{.1cm}}
\section*{Acknowledgments}

VP is grateful to Pierre-Guy Plamondon and Yann Palu for various working sessions which led to the presentation of \cref{sec:representationTheory}.


\bibliographystyle{alpha}
\bibliography{separatingTrees}
\label{sec:biblio}

\end{document}